\documentclass[12pt]{article}

\usepackage{amsthm}
\usepackage{amssymb}
\usepackage{amsmath}
\usepackage{amsfonts}
\usepackage{times}

\usepackage[paper=a4paper, top=2.5cm, bottom=2.5cm, left=2.5cm, right=2.5cm]{geometry}

\def\qed{\hfill $\vcenter{\hrule height .3mm
\hbox {\vrule width .3mm height 2.1mm \kern 2mm \vrule width .3mm
height 2.1mm} \hrule height .3mm}$ \bigskip}

\def \RR {\mathbb R}

\def \EE {\mathbb E}

\def \PP {\mathbb P}
\def \eps {\varepsilon}

\def \Rnn {\mathbb{R}^{n-1}}
\def \TTT {\mathcal{T}}

\def \FF {\mathcal{F}}
\def \St {\mathcal{S}}
\def \EPS {\delta}

\newtheorem{theorem}{Theorem}
\newtheorem{lemma}[theorem]{Lemma}

\newtheorem{fact}[theorem]{Fact}
\newtheorem{claim}[theorem]{Claim}
\newtheorem{conjecture}[theorem]{Conjecture}
\newtheorem{proposition}[theorem]{Proposition}
\newtheorem{corollary}[theorem]{Corollary}
\theoremstyle{definition}

\theoremstyle{remark}
\newtheorem{remark}[theorem]{Remark}

\long\def\symbolfootnotetext[#1]#2{\begingroup
\def\thefootnote{\fnsymbol{footnote}}\footnotetext[#1]{#2}\endgroup}

\begin{document}

\author{Ronen Eldan}
\title{A two-sided estimate for the Gaussian noise stability deficit}
\date{}
\maketitle
\begin{abstract}
The Gaussian noise-stability of a set $A \subset \RR^n$ is defined by
$$
\St_\rho(A) = \PP \left ( X \in A ~ \& ~ Y \in A  \right )
$$
where $X,Y$ are standard jointly Gaussian vectors satisfying  $\EE [X_i Y_j] = \delta_{ij} \rho$. Borell's inequality states that for all $0 < \rho < 1$,  among all sets $A \subset \RR^n$ with a given Gaussian measure, the quantity $\St_\rho(A)$ is maximized when $A$ is a half-space.

We give a novel short proof of this fact, based on stochastic calculus. Moreover, we prove an almost tight, two-sided, dimension-free robustness estimate for this inequality: by introducing a new metric to measure the distance between the set $A$ and its corresponding half-space $H$ (namely the distance between the two centroids), we show that the deficit $\St_\rho(H) - \St_\rho(A)$ can be controlled from both below and above by essentially the same function of the distance, up to logarithmic factors.

As a consequence, we also establish the conjectured exponent in the robustness estimate proven by Mossel-Neeman, which uses the
total-variation distance as a metric. In the limit $\rho \to 1$, we obtain an improved dimension-free robustness bound for the Gaussian isoperimetric inequality. Our estimates are also valid for a generalized version of stability where more than two correlated vectors are considered.

\end{abstract}

\section{Introduction}

The topic of this paper is the \emph{isoperimetric inequality} in \emph{Gauss space} and an extension of this inequality, referred to as the \emph{Gaussian noise stability} inequality. The Gauss space is the Euclidean space $\RR^n$ equipped with the standard Gaussian measure $\gamma^n$ defined by
$$
\gamma^n(A) = \frac{1}{ (2 \pi)^{n/2} }\int_{A} e^{-|x|^2/2} dx.
$$
In the following, we will often abbreviate $\gamma=\gamma^n$. The Isoperimetric inequality, initially proved independently by Sudakov-Tsirelson (\cite{ST}) and Borell (\cite{Bor1}) states that among all subsets of $\RR^n$ of a given Gaussian measure, the sets minimizing the Gaussian surface area (defined as the integral of the Gaussian density with respect to the $(n-1)$-dimensional Hausdorff measure on the boundary of the set) are half-spaces. More recent proofs based on probabilistic, analytic, geometric and discrete methods can be found in \cite{G1,G2,G3,G4,G5}. The case of equality has been settled in \cite{G6} which further extends the methods introduced in \cite{G2}.

An extension of this inequality due to C. Borell (\cite{Bor2}), who introduced the notion of Gaussian \emph{noise stability}, states that half-spaces are the most \emph{stable} sets. The usual definition of stability of a set is the probability that two standard Gaussian vectors with a given correlation both lie in the set. For a more precise definition, let $X,Y$ and $Y'$ be independent standard Gaussian vectors in $\RR^n$ and let $0 \leq \rho \leq 1$. We define the \emph{Gaussian noise stability} with parameter $\rho$ of a measurable set $A \subset \RR^n$ by
$$
\St_\rho(A) = \PP \left ( \sqrt{\rho} X + \sqrt{1-\rho} Y \in A \mbox { and } \sqrt{\rho} X + \sqrt{1-\rho} Y' \in A \right ).
$$
Borell's theorem states that if $A,H \subset \RR^n$ are such that $H$ is a half-space and $\gamma(A) = \gamma(H)$ then $\St_\rho(H) \geq \St_\rho(A)$ for all $0 \leq \rho \leq 1$. It turns out that this result admits diverse applications in numerous fields such as approximation theory, high-dimensional phenomena and rearrangement inequalities, and recently some surprising applications to discrete analysis and game theory have also been found (see \cite{MN} and references therein). Alternative proofs of Borell's inequality were given by Isaksson-Mossel and  Kindler-O`Donnell \cite{S1,S2} and recently Mossel and Neeman (\cite{MN}) found a semi-group proof which also settles the equality case.

Both the Gaussian isoperimetric inequality and the noise stability inequality admit so-called \emph{robustness} estimates. As the original inequalities only claim that the minimum (or maximum) of a certain quantity is attained on half-spaces, a robustness estimate also quantifies the \emph{deficit} in these inequalities in terms of the distance, under a certain metric, of the set from being a half-space. For example, one may try to prove that if $A,H \subset \RR^n$ satisfy $\gamma(H) = \gamma(A)$ and if the Gaussian surface area of $A$ differs from that of $H$ by a small number $\delta>0$, then there necessarily exists a half-space whose total-variation distance to the set $A$ is smaller than some function of $\delta$ and $\gamma(A)$ (and maybe of the dimension) which goes to zero as $\delta \to 0$. A robustness estimate of the noise-stability inequality will do the same where the deficit $\delta = \St_\rho(H) - \St_\rho(A)$ is considered (and in this case, the distance to the half-space can also be a function of $\rho$).

The first robustness estimate for the Gaussian isoperimetric inequality was proven by Cianchi-Fusco-Maggi-Pratelli in \cite{CFMP}, and is based on a geometric approach. Mossel and Neeman found a different proof based on a more analytic approach, and provided a dimension-free estimate. In the more recent paper \cite{MN}, Mossel and Neeman managed to prove a robustness result for the Gaussian noise-stability and in the limit case $\rho \to 1$ they also attain an improvement of their isoperimetric robustness. The metric used in all of these estimates is the \emph{total- variation} distance between the Gaussian measure restricted to the set and to the corresponding half-space. We would also like to mention a recent result of Bobkov, Gozlan, Roberto and Samson \cite{BGRS} who prove a somewhat-related robust logarithmic-Sobolev inequality, in which the deficit is estimated by transport and information-theoretic distances.

Another possible extension of the noise stability inequality, first explicitly introduced by E. Mossel is the following: for a number $q>1$ and for $0 < \rho < 1$, we refer to the following as the $q$-stability of $A$:
$$
\St^q_\rho(A) = \EE \Bigl [ \PP( \sqrt{\rho} X + \sqrt{1-\rho} Y \in A ~ | X )^q  \Bigr ]
$$
where $X$ and $Y$ are independent standard Gaussian vectors. Evidently, $\St_\rho(A) = \St^2_\rho(A)$. When $q$ is an integer, this quantity can be thought of as the probability of $q$ Gaussian vectors whose mutual correlation is $\rho$ to all be inside $A$. It was initially proven by Isaksson and Mossel (\cite{S1}) that half-spaces are the maximizers of this expression when constraining on the Gaussian volume. The equality case and a robustness bound for this extension has been established by J. Neeman (\cite{N}). \\

This note has a few objectives. First, we present a novel proof of the Gaussian noise stability inequality, based on stochastic calculus. As a consequence, in the limit case we derive a new proof of the Gaussian isoperimetric inequality. Our proof is relatively short, and also provides a very simple argument for establishing the equality case. Moreover, our proof is also valid for the more general $q$-stability defined above.

Next, we introduce a new metric $\eps(A)$ to measure the distance between the set $A$ and its closest-possible half-space, say $H$. This metric, defined roughly as the distance between the corresponding centroids, turns out to be rather natural in this context: We prove that up to constants that depend only on $\rho$ and on $\gamma(A)$, the deficit $\St_\rho(H) - \St_\rho(A)$ can be bounded from both below and above by the same power of $\eps(A)$, with only a logarithmic correction. The lower bound of the estimate we obtain is valid also for the more-general $q$-stability. As a corollary, we improve the dimension-free robustness result of \cite{MN} which uses the total variation metric, getting the best possible exponent, up to a logarithmic factor.

Our estimate also has an optimal dependence on the parameter $\rho$. Thanks to this fact we are able, as a limit case, to derive a robustness estimate for the Gaussian isoperimetric inequality, thus obtaining a dimension-free estimate with an optimal exponent (again, up to a logarithmic term). \\

Let us begin with some definitions, towards the formulation of our results. For a measurable $A \subset \RR^n$ whose Gaussian centroid is not the origin, we define
\begin{equation} \label{vA}
v(A) = \int_{A} x d \gamma(x)
\end{equation}
(otherwise, if the above integral is zero, we arbitrarily take $v(A)=e_1$ for some fixed unit vector $e_1 \neq 0$). Let $H(A)$ be the half-space of the form
\begin{equation} \label{defH}
H(A) = \left  \{x; ~ \left \langle x, v(A) \right \rangle \geq \alpha \right  \}
\end{equation}
where $\alpha$ is chosen such that $\gamma(H(A)) = \gamma(A)$. In other words, $H(A)$ is the half-space whose Gaussian measure is the same as that of $A$ and whose Gaussian center of mass is the closest possible to the Gaussian center of mass of $A$. \\

In our first theorem, the inequality is due to C. Borell and the characterization of the equality case (in the case $q=2$) is due to Mossel-Ne`eman. Our main contribution here is giving a rather short and simple proof based on new methods.

\begin{theorem} \label{thm1}
For all measurable subsets $A \subset \RR^n$, for all $0 \leq \rho < 1$ and $q>1$, one has
$$
\St^q_\rho(H(A)) \geq \St^q_\rho(A).
$$
Moreover, equality holds if and only if the symmetric difference between $A$ and $H(A)$ has measure zero.
\end{theorem}

The reader who is interested in the proof of this theorem may directly skip to the beginning of section 2. \\ \\

We move on to the robustness estimate. Before we can formulate it, we need a few more definitions. Define, for all measurable $B \subset \RR^n$,
\begin{equation} \label{defQ}
q(B) = \left | \int_{B} x d \gamma(x) \right |.
\end{equation}
We measure the distance between a set and its corresponding half-space using the following metric,
$$
\eps(A) := q(H(A))^2 - q(A)^2.
$$
In other words, up to a factor which depends on the Gaussian measure of $A$, the metric we consider is the distance between the centroid of $A$, and that of the half-space closest to it (we will see below that this quantity is always non-negative). \\ \\
Our two-sided robustness estimate, the main theorem of this paper, reads

\begin{theorem} \label{thmrobust}
For every $q > 1$ and $0 < s < 1$, there exist constants $C_s, c_s, c'_{s,q} > 0$ such that the following holds: Let $0 < \rho < 1$ and let $A \subset \RR^n$ be a measurable set satisfying $\eps(A) < e^{-1/\rho}$ and $0 < \gamma(A) < 1$. Then
\begin{equation} \label{robustineq}
c_{\gamma(A)} \eps(A) |\log(\eps(A))|^{-1} \sqrt{1 - \rho} \leq \St_\rho(H(A)) - \St_\rho(A) \leq  \frac{C_{\gamma(A)}}{\sqrt{1 - \rho}} \eps(A)
\end{equation}
and
\begin{equation} \label{robustineq2}
c'_{\gamma(A), q} \eps(A) |\log(\eps(A))|^{-1} \sqrt{1 - \rho} \leq \St^q_\rho(H(A)) - \St^q_\rho(A).
\end{equation}
\end{theorem}
The reader is referred to the next subsection for a discussion about the sharpness of the this result and about possible extensions. \\

Arguably, for many explicit examples of sets $A$, the quantity $\eps(A)$ is significantly easier to calculate than the actual noise stability of $A$, as it only depends on the density of the projection of the set onto a one-dimensional subspace, hence the calculation boils down to computing one-dimensional integrals. In light of the above theorem, one can obtain an approximation of the latter by the former. 

\begin{remark}
This theorem has another interpretation in terms of the Fourier-Hermite representation of the Ornstein-Uhlenbeck noise operator. For a function $f \in L_2(\gamma)$, consider its Fourier-Hermite representation
$$
f(x) = \sum_{\ell \in \mathbb{N}^n} C_\ell(f) H_\ell (x)
$$ 
where $\{H_\ell\}_{\ell \in \mathbb{N}^n}$ is the orthonormal basis of Hermite polynomials of $L_2(\RR^n, \gamma)$ and $C_\ell(f) := \langle f, H_\ell\rangle_{L_2(\gamma)}$ are the corresponding Hermite coefficients. It is well known that
\begin{equation} \label{Fourier-Hermite}
\mathcal{S}_\rho (A) = \sum_{ \ell \in \mathbb{N}^n} \rho^{|\ell|} C_\ell(\mathbf{1}_A)^2
\end{equation}
where for $\ell \in \mathbb{N}^n$ we write $|\ell| = \sum_{i=1}^n \ell_i$. Moreover, by definition we have that
$$
\eps(A) = \sum_{ \ell \in \mathbb{N}^n \atop |\ell| = 1 } \left ( C_\ell(\mathbf{1}_{H(A)})^2 - C_\ell(\mathbf{1}_A)^2 \right )
$$
(hence our metric can be seen as the energy deficit of the first order Fourier-Hermite coefficients). Define $g(x) = x |\log x|^{-1}$. Then inequality \eqref{robustineq} takes the form
$$
c_{\gamma(A)} g \left (\sum_{ \ell \in \mathbb{N}^n \atop |\ell| = 1 } \left ( C_\ell(\mathbf{1}_{H(A)})^2 - C_\ell(\mathbf{1}_A)^2 \right ) \right ) \sqrt{1 - \rho} 
$$
$$
\leq  \sum_{ \ell \in \mathbb{N}^n} \rho^{|\ell|} \left ( C_\ell(\mathbf{1}_{H(A)})^2 - C_\ell(\mathbf{1}_A)^2 \right ) 
$$
$$
\leq  \frac{C_{\gamma(A)}}{\sqrt{1 - \rho}} \left (\sum_{ \ell \in \mathbb{N}^n \atop |\ell| = 1 } \left ( C_\ell(\mathbf{1}_{H(A)})^2 - C_\ell(\mathbf{1}_A)^2 \right ) \right ).
$$
\end{remark}

As a first corollary to this theorem, we also get a robustness estimate in terms of the \emph{total-variation} metric between a set and its corresponding half-space. For two sets $A,B$ we denote by $A \Delta B$ the symmetric difference between two sets, and define
$$
\delta(A) := \gamma( A \Delta H(A) ).
$$
We have,

\begin{corollary} \label{cortv}
For every $q > 1$ and $0 < s < 1$, there exists a constant $c_{s,q} > 0$ such that the following holds: Let $0 < \rho < 1$ and let $A \subset \RR^n$ be a measurable set satisfying $\delta(A) < e^{-1/\rho}$ and $0 < \gamma(A) < 1$. Then
\begin{equation} \label{eqcortv}
c_{\gamma(A), q} \delta(A)^2 |\log(\delta(A))|^{-1} \sqrt{1 - \rho} \leq \St^q_\rho(H(A)) - \St^q_\rho(A).
\end{equation}
\end{corollary}

A second corollary to the robustness estimate for the noise stability is the limit case for $\rho \to 1$, namely a robustness estimate for the Gaussian isoperimetric inequality. For a measurable set $A \subset \RR^n$, define the Gaussian surface area of $A$ by
$$
\mathcal{O}^n(A) := \limsup_{\eps \to 0} \eps^{-1} \gamma(A_\eps \setminus A)
$$
where
$$
A_\eps := \left \{ x \in \RR^n; ~~ \exists y \in A \mbox{ such that } |x-y| \leq \eps  \right \}
$$
is the $\eps$-extension of $A$.

The next result is a direct corollary of Theorem \ref{thmrobust} following a method introduced by M. Ledoux, about which the author learned from \cite{MN}. By plugging the result of the theorem and of the above corollary into Ledoux's method, described in \cite{Le} (in the discussion following proposition 8.5) we immediately get the next corollary. 

\begin{corollary}
There exists a universal constant $c>0$ and, for every $0 < s < 1$, a constant $c_s > 0$ such that the following holds: for all measurable $A \subset \RR^n$ with $\eps(A) < c$, one has
\begin{equation} \label{sabound}
\mathcal{O}^n(A) - \mathcal{O}^n(H(A)) \geq c_{\gamma(A)} \eps(A) |\log(\eps(A))|^{-1}
\end{equation}
and for all measurable $A \subset \RR^n$ such that $\delta(A) < c$, one has
$$
\mathcal{O}^n(A) - \mathcal{O}^n(H(A)) \geq c_{\gamma(A)} \delta(A)^2 |\log(\delta(A))|^{-1}.
$$
\end{corollary}

\begin{remark}
By combining the upper and lower bounds of Theorem \ref{thmrobust} we may also get non-trivial relations between the quantities $\St_\rho(A)$ at different values of $\rho$, as well as bounds on the noise stability of a set in terms of its surface area. As an example, if $A \subset \RR^n$ and $0 < \rho < 1$ are such that $\eps(A) < e^{-1/\rho}$, then by combining (\ref{sabound}) with the upper bound in (\ref{robustineq}) one gets
$$
\mathcal{O}^n(A) - \mathcal{O}^n(H(A)) \geq c_{\gamma(A)} \sqrt{1 - \rho} \St_\rho(A) \left | \log \left ( \St_\rho(A) \sqrt{1 - \rho} \right ) \right |^{-1}.
$$
where $c_{\gamma(A)} > 0$ depends only on $\gamma(A)$.
\end{remark}

The structure of this paper is as follows. In section 2 we prove Theorem \ref{thm1} and in section 2 we prove Theorem \ref{thmrobust} and Corollary \ref{cortv}. Section 3 is an appendix in which we tie up some loose ends.

\bigskip
In the rest of the note, we use the following notation. The constants $C,C',c,c'$ will denote positive universal constants whose values may change between appearances in different formulae. We define $\gamma^k:\RR^k \to \RR$ the density of the standard Gaussian measure on $\RR^k$ and by slight abuse of notation we define $\gamma^k(A)$ to be the Gaussian measure of the set $A \subset \RR^k$. We abbreviate $\gamma = \gamma^n$, with $n$ being a fixed dimension all through the note. For two sets $A,B$ we define by $A \Delta B$ the symmetric difference between them. For a positive semi-definite symmetric matrix $A$, we denote its largest eigenvalue by $||A||_{OP}$. For any matrix $A$, we denote the sum of its diagonal entries by $Tr(A)$, and by $||A||_{HS}^2$ we denote the sum of the eigenvalues of the matrix $A^T A$. For a random vector $X$, we denote $\mathrm{Cov}(X) = \EE[X \otimes X] - \EE[X] \otimes \EE[X]$, the covariance matrix of $X$. Finally, for a continuous time stochastic process $X_t$ adapted to a filtration $\mathcal{F}_t$, we denote by $[X]_t$ the quadratic variation of $X_t$ between time $0$ and $t$. For a pair of continuous time stochastic processes $X_t, Y_t$, the quadratic covariation will be denoted by $[X, Y]_t$. By $d X_t$ we denote the It\^{o} differential of $X_t$, which we understand as a pair of predictable processes $(\sigma_t, \mu_t)$ such that $X_t$ satisfies  stochastic differential equation $d X_t = \sigma_t d B_t + \mu_t dt$ where $B_t$ is a Brownian motion. We also denote
$$
\Phi(s) = \frac{1}{\sqrt{2 \pi}} \int_{-\infty}^{s} e^{-x^2/2} dx
$$
the Gaussian cumulative distribution function, and write $\Psi(s) = \Phi^{-1}(s)$. \\ \\

\emph{Acknowledgments} ~ I am grateful to Elchanan Mossel for inspiring me to work on this problem and for several fruitful discussions in which, in particular, he suggested that the method should work for $q$-stability and for the isoperimetric problem. I am deeply thankful to Bo'az Klartag for a very useful discussion in which he gave me the idea of using Talagrand's theorem in the proof of Lemma \ref{estup}. Finally, I thank Gil Kalai for introducing me to this topic and Yuval Peres, Joe Neeman, Joseph Lehec and James Lee for useful comments on a preliminary version of this note.

\subsection{Discussion}
Before we move on to the proofs, we would like to discuss the optimality of our estimates and suggest possible future research.

First, consider the robustness inequality (\ref{robustineq}). It is easy to see
that the dependence of the upper bound on $\eps(A)$ is tight: for $\eps \geq 0$
define
$$
A_\eps = (-\infty, \Psi(1/2-\eps)] \cup [\Psi(3/4), \Psi(3/4 + \eps)],
$$
where $\Psi(x) = \Phi^{-1}(x)$ is the inverse Gaussian cumulative distribution function.
We claim that this set demonstrates that the upper bound is tight. It is easy to check that $\eps(A) \sim \eps$ and that if $X,Y$ are Gaussian variables whose correlation is $0.01 < \rho < 0.99$ then
$$
\PP(Y \in (-\infty, \Psi(1/2-\eps)) ~ | X \in [\Psi(1/2-\eps), 0] ) >
$$
$$
(1+c) \PP(Y \in (-\infty, \Psi(1/2-\eps)) ~ | X \in [\Psi(3/4), \Psi(3/4 + \eps)] ) \geq (1+c) c'
$$
for all $\eps < 1/4$ where $c,c'>0$ are constants which do not depend on $\eps$. Therefore
$$
\St_\rho(H(A)) = \St_\rho ((-\infty, \Psi(1/2 - \eps)]) + 2 \eps \PP(Y \in (-\infty, \Psi(1/2-\eps)] ~ | X \in [\Psi(1/2-\eps), 0] ) + o(\eps)
$$
and
$$
\St_\rho(A) = \St_\rho ((-\infty, \Psi(1/2 - \eps)]) + 2 \eps \PP(Y \in (-\infty, \Psi(1/2-\eps)] ~ | X \in [\Psi(3/4), \Psi(3/4 + \eps)] ) + o(\eps) 
$$
so
$$
\St_\rho(H(A)) - \St_\rho(A) > 2 c' c  \eps + o(\eps).
$$
This shows that the dependence of the upper bound on $\eps$ cannot be improved. One could hope that the logarithmic term in the lower bound is fully removed thus obtaining a tight bound. Alas, if we define
$$
A_\eps = (-\infty, \Psi(1/2-\eps)] \cup [\Psi(1 - \eps), \infty)
$$
then it is not hard to see that $\eps(A) \sim \eps \sqrt{|\log(\eps)|}$. On the other hand, we have
$$
\St_\rho(H(A_\epsilon)) \leq \St_\rho ((-\infty, \Psi(1/2 - \eps)]) + O(\eps)
$$
and
$$
\St_\rho(A_\epsilon) \geq \St_\rho ((-\infty, 1/2 - \eps]).
$$
It follows that
$$
\St_\rho(H(A_\epsilon)) - \St_\rho(A_\epsilon) = O(\epsilon)
$$
so at least a term of the order $\sqrt{|\log \eps(A)|}$ is necessary. \\

It seems from the proofs that this type of "tail" phenomenon might be the only reason for which the logarithmic term is needed. We would like to formulate a conjecture suggesting that a slightly perturbed metric could provide a tight bound. To define this metric we write $v = v(H(A)) / |v(H(A))|$ (where $v(\cdot)$ is defined in equation (\ref{vA})), and let $\mu$ and $\nu$ be the push-forward under the map $x \to \langle v, x \rangle$ of the Gaussian measure restricted to the sets $A$ and $H(A)$ respectively. Denote by $f(x)$ and $g(x)$ the corresponding densities of $\mu$ and $\nu$ with respect to $\gamma^1$.

Inspired by equation (\ref{inspiration}) below, we define
$$
\tilde \eps_\rho (A) = \left | \int_{\RR} \Phi \left ( \frac{\rho x - \alpha}{\sqrt{1-\rho^2}}  \right ) (g(x) - f(x)) d \gamma^1(x) \right |.
$$
\begin{conjecture}
For every $0 < s < 1$, there exist constants $C_s, c_s > 0$ such that the following holds: Let $0 < \rho < 1$ and let $A \subset \RR^n$ be a measurable set satisfying $\eps(A) < e^{-1/\rho}$ and $0 < \gamma^n(A) < 1$. Then
\begin{equation}
C_{\gamma(A)} \tilde \eps_\rho(A) \geq \St_\rho(H(A)) - \St_\rho(A) \geq c_{\gamma(A)} \tilde \eps_\rho(A) (1 - \rho).
\end{equation}
In particular, the expression $\St_\rho(H(A)) - \St_\rho(A)$ is equivalent, up to constants depending only on $\rho$ and $\gamma(A)$, to an expression depending only on the marginal of the set $A$ on the direction $v$.
\end{conjecture}
\bigskip
Finally, let us discuss the optimality of Corollary \ref{cortv}. We claim that the exponent $2$ of the expression $\delta(A)^2$ appearing in equation (\ref{cortv}) cannot be improved. Consider the example
$$
B_\eps = (-\infty, \Psi(1/2-\eps)] \cup [\Psi(1/2 + \eps), \Psi(1/2 + 2\eps) ].
$$
It is easy to verify that $\eps(B) \sim \eps^2$ while $\delta(B) = \eps$. It follows
from the upper bound in Theorem \ref{thmrobust} that as $\eps \to 0$, the dependence of the deficit on $\delta(A)$ is correct, maybe up to the logarithmic factor. We conjecture that the logarithmic factor in this corollary can be removed.

\section{Proof of Theorem \ref{thm1}} \label{sec2}

This section is dedicated to the proof of Theorem \ref{thm1}, and is divided into three parts:
In the first part we introduce a stochastic process associated with the set $A \subset \RR^n$, upon the analysis of which the proof is based. The second part is an overview of the main steps of our proof and in the third part we will provide the detailed argument.

\subsection{The process $S_t$}

First define for all $v \in \RR^n$ and $\sigma > 0$,
$$
\gamma_{v, \sigma} (x) = \frac{1}{\sigma^n (2 \pi)^{n/2}} \exp \left (- \frac{1}{2 \sigma^2} |x - v|^2 \right ),
$$
the density of the Gaussian centered at $v$ with covariance matrix $\sigma^2 \mathrm{Id}$, and abbreviate $\gamma(x) = \gamma_{0,1}(x)$. Let $X$ be a standard Gaussian random vector in $\RR^n$. \\ \\
Given a measurable set $A \subset \RR^n$, our goal is to analyse the quantity
$$
\St^q_\rho(A) = \EE_X \left [ \left ( \int_A  \gamma_{\sqrt{\rho}X, \sqrt{1 - \rho}}(x) dx \right )^q \right ].
$$
Instead of considering the vector $X$, let $W_t$ be a standard Brownian motion in $\RR^n$, adapted to a filtration $\FF_t$, with an underlying probability space $(\Omega_1, \Sigma_1, P_1)$. Clearly $\sqrt{\rho} X$ has the same distribution of $W_\rho$, and therefore
$$
\St^q_\rho(A) = \EE \left [ \left ( \int_A  \gamma_{W_\rho, \sqrt{1 - \rho}}(x) dx \right )^q \right ].
$$
The main idea of the proof is to consider the process
$$
S_t = \PP(W_1 \in A~ | \FF_t).
$$
Since conditioned on $\FF_t$, the vector $W_1$ is a Gaussian vector with expectation $W_t$ and covariance matrix $\sqrt{1-t} \mathrm{Id}$, we have 
$$
S_t = \int_A  \gamma_{W_t, \sqrt{1 - t}}(x) dx, ~~ \forall 0 \leq t < 1,
$$
so that our quantity of interest becomes
$$
\St_\rho^q(A) = \EE[S_\rho^q].
$$
\subsection{Overview of the main steps}

For simplicity, in this subsection we will consider the basic notion of noise stability, assuming that $q=2$. The last formula becomes,
$$
\St_\rho(A) = \EE[S_\rho^2].
$$
Since the process $S_t$ is, by definition, a martingale, we have by the It\^{o} isometry
$$
\St_\rho(A) = S_0^2 + \EE [S]_\rho
$$
where $[S]_t$ denotes the quadratic variation of the process $S_t$. Noting that $S_t$ is a smooth function of $W_t$, a straightforward calculation using It\^{o}'s formula gives
\begin{equation} \label{dQVint}
d [S]_t = (1-t)^{-1} q^2(A_t) dt
\end{equation}
where $q(\cdot)$ is the function defined in equation \eqref{defQ} and $A_t = \frac{A - W_t}{\sqrt{1-t}}$. Moreover, it turns out that $\gamma(A_t) = S_t$.

The last formula suggests that the time derivative of the quadratic variation depends on the distance of the Gaussian center of mass of a set, whose measure is $S_t$, from the origin. Formula \eqref{Fourier-Hermite} may shed some light on this fact: according to this formula, we see that the quantity $\frac{d}{d \rho} \St_\rho(A) |_{\rho = 0}$ is indeed equal to the energy of the first order Fourier-Hermite coefficients of the function $\mathbf{1}_A(\cdot)$, which is exactly $q(A)$. The set-valued process $A_t$ can be thought of as the evolution of the set $A$ as seen by the measures $W_1 | \FF_t$ (more accurately, $A_t$ is the set $A$ under a linear transformation which pushes forward the measure of $W_1$ conditioned on $\FF_t$ to the standard Gaussian measure). 

Since we would ultimately like to compare the noise stability of $A$ with that of the half-space $H(A)$, we consider the analogous process
$$
Q_t = \PP \bigl (\tilde W_1 \in H(A) | \tilde \FF_t \bigr)
$$
where $\tilde {W_t}$ is a standard Brownian motion adapted to a filtration $\tilde \FF_t$, so that 
$$
\St_\rho(H(A)) = S_0^2 + \EE [Q]_\rho.
$$
By carrying the same calculation as above, we get 
\begin{equation} \label{dQVint2}
d [Q]_t = (1-t)^{-1} q^2(\tilde A_t) dt
\end{equation}
where $\tilde A_t = \frac{H(A) - \tilde {W_t}}{\sqrt{1-t}}$ is a half-space whose Gaussian measure is $Q_t$. 

The proof of the theorem is reduced to showing the inequality
\begin{equation} \label{nts111}
\EE \left [ \int_0^\rho (1-t)^{-1} q^2(A_t) dt \right ] \leq \EE \left [ \int_0^\rho (1-t)^{-1} q^2(\tilde  A_t) dt \right ].
\end{equation}

At this point, we make the following simple geometric observation: among all sets $B \subset \RR^n$ with a prescribed Gaussian measure $0<\alpha<1$, the set which minimizes the norm of its Gaussian center of mass is a half-space (Claim \ref{claim1} below). In other words, we have for all measurable $B \subset \RR^n$,
\begin{equation} \label{ineqclaim1}
q(B) \leq q(H(B)).
\end{equation}

Na\"{i}vely, one could hope that applying the inequality \eqref{ineqclaim1} directly to the sets $A_t$ and $\tilde A_t$ (which are indeed defined as the evolutions of the set $A$ and its corresponding half-space) could establish equation \eqref{nts111}. However, note that in order to do this, we must somehow have that $\gamma(A_t) = \gamma(\tilde A_t)$, or in other words, we need to somehow couple the processes $W_t$ and $\tilde W_t$ so that $S_t = Q_t$. Alas, this is impossible to achieve. However, it gives rise to the idea of considering the two processes under the change of time described next.

The next step is where the theory of stochastic calculus plays the most crucial role. Since both processes $S_t$ and $Q_t$ are martingales, by applying the Dambis / Dubins-Schwartz theorem, we may consider each of them as a time-changed Brownian motion. Moreover, an application of the disintegration theorem will allow us to couple between those Brownian motions in the following way: we will prove that there exists a probability space containing both processes $S_t, Q_t$ and a another process $B(T)$ such that $B(T) - S_0$ is a Brownian motion and such that
$$
S_t = B([S]_t), ~~ Q_t = B([Q]_t), ~~ \forall 0 \leq t \leq 1.
$$
Defining $\tau_1(T)$ and $\tau_2(T)$ as the inverse functions of $t \to [S]_t$ and $t \to [Q]_t$ respectively, according to the last equation we have $S_{\tau_1(T)} = Q_{\tau_2(T)}$ which implies that $\gamma(A_{\tau_1(T)}) = \gamma(\tilde A_{\tau_2(T)})$ for all $0 < T < [S]_1$. Finally, we can use equation \eqref{ineqclaim1} with $B = A_{\tau_1(T)}$ to assert that
$$
q(A_{\tau_1(T)}) \leq q \bigl (\tilde A_{\tau_2(T)} \bigr ).
$$
Together with the formulas for $d [S]_t$ and $d [Q]_t$ (formulas \eqref{dQVint} and \eqref{dQVint2} above), we get that
$$
\frac{d}{d T} \tau_1(T) \geq \frac{d}{d T} \tau_2(T), ~~ \forall 0 < T < [S]_1.
$$
This immediately implies that $[S]_\rho \leq [Q]_\rho$ almost surely, which completes the proof of the inequality. The analysis of the equality case will be straightforward.

\subsection{The proof}

We begin with the following lemma, which will be helpful to us in calculating It\^{o} differentials related to the process $S_t$.
\begin{lemma} \label{LemdiffF}
Denote
$$
F_t(x) = \gamma_{W_t, \sqrt{1-t}}(x).
$$
For every $x \in \RR^n$ the process $F_t(x)$ is a local martingale satisfying the stochastic differential equation
\begin{equation} \label{stochastic}
F_0(x) = \gamma(x), ~~ d F_t(x) = (1-t)^{-1} F_t(x) \langle x - W_t, d W_t \rangle.
\end{equation}
Moreover, for any measurable function $\phi:\RR^n \to \RR$ satisfying $|\phi(x)| < C_1 + C_2 |x|^p$ for some constants $C_1,C_2,p>0$, we have that the process $t \to \int_{\RR^n} \phi(x) F_t(x) dx$ is a martingale with respect to the filtration $\FF_t$, which satisfies
\begin{equation} \label{stochasticfub}
d \int_{\RR^n} \phi(x) F_t(x) dx = (1-t)^{-1} \left \langle \int_{\RR^n} \phi(x) (x - W_t) F_t(x) dx, d W_t \right \rangle.
\end{equation}
\end{lemma}

The proof, which  is a straightforward calculation, is postponed to the appendix. \\
\begin{remark}
Equation (\ref{stochastic}) could be seen as a stochastic evolution equation on the space of Gaussian densities. 
Equations of a similar nature, where the initial density is an arbitrary function seem to be rather useful tool for proving concentration inequalities,
as demonstrated in \cite{E1,E2,E3}.
\end{remark}

Using the notation of the lemma, we have
\begin{equation} \label{defSt}
S_t = \int_A F_t(x) dx, ~~ \forall 0 \leq t < 1.
\end{equation}
We can now calculate, using equation \eqref{stochasticfub},
\begin{equation} \label{dst}
d S_t = d \int_A F_t(x) dx =
\end{equation}
$$
(1-t)^{-1} \left \langle \int_{A}  (x - W_t) F_t(x) dx , d W_t \right \rangle =
$$
$$
\frac{(1-t)^{-1}}{(2 \pi (1-t))^{(n/2)}}  \left \langle  \int_{A} (x - W_t) \exp \left (- \frac{1}{2(1-t)} |x - W_t|^2 \right ) dx, d W_t \right  \rangle =
$$
(substituting $y =  \frac{x - W_t}{ \sqrt{1 - t} }$)
$$
(1-t)^{-1/2} \left \langle \int_{ \frac{A - W_t}{\sqrt{1-t}}} y \gamma(y) dy , d W_t \right  \rangle.
$$
Recall that for all measurable $B \subset \RR^n$, we define
$$
q(B) = \left | \int_{B} x \gamma(x) dx \right |
$$
and write
$$
A_t = \frac{A - W_t}{\sqrt{1-t}}.
$$
Under this notation, equations \eqref{defSt} and \eqref{dst} become
\begin{equation} \label{defst2}
S_t = \int_{A_t} \gamma(x) dx = \gamma(A_t)
\end{equation}
and
\begin{equation} \label{dSt}
d S_t = (1-t)^{-1/2} \left \langle \int_{A_t} x d \gamma(x), d W_t \right \rangle.
\end{equation}
The last equation also gives
\begin{equation} \label{quadvar}
d [S]_t = (1-t)^{-1} q^2(A_t) dt
\end{equation}
where $[S]_t$ denotes the quadratic variation of the process $S_t$. \\ \\
In particular, we see that $S_t$ is an It\^{o} process, so thanks to It\^{o}'s formula,
$$
d S_t^q = q S_t^{q-1} d S_t + \frac{1}{2} q(q-1) S_t^{q-2} d [S]_t.
$$
Since $S_t$ is a bounded martingale, by integrating the last equation and taking expectation we get
\begin{equation} \label{needtoshow}
\St^q_\rho(A) =  \EE[S_\rho^q] = S_0^q + \EE \left [ \int_0^\rho d S_t^q \right ] = S_0^q + \frac{1}{2} q(q-1) \EE \left [ \int_0^\rho S_t^{q-2} d [S]_t \right ]
\end{equation}
and, in particular, by taking $q=2$ we see that $\St_\rho(A) = \gamma(A)^2 + \EE [S]_\rho$. \\ \\

Our goal is to compare $\St_\rho^q(A)$ with $\St_\rho^q(H(A))$. To that end, we want to define the process $Q_t$ to be an analogous process to $S_t$ where the initial set $A$ is replaced by its corresponding half-space $H(A)$. In other words, we define
\begin{equation} \label{defQt}
Q_t = \PP \left . \left (\tilde W_1 \in H(A) ~ \right | \tilde \FF_t \right ) = \int_{H(A)}  \gamma_{\tilde W_t, \sqrt{1 - t}}(x) dx,
\end{equation}
where $\tilde W_t$ is a standard Brownian motion in $\RR^n$, adapted with to a filtration $\tilde F_t$. At this point, we consider $\tilde W_t$ to be a process defined over a different probability space, which we denote by $(\Omega_2, \Sigma_2, P_2)$.

In analogy with (\ref{needtoshow}) we have
\begin{equation} \label{needtoshow2}
\St^q_\rho(H(A)) = \EE[Q_\rho^q] = Q_0^q + \frac{1}{2} q(q-1) \EE \left [ \int_0^\rho Q_t^{q-2} d [Q]_t \right ]
\end{equation}
and since $Q_0 = \gamma(H(A)) = \gamma(A) = S_0$, the proof is reduced to showing that,
\begin{equation} \label{eq2}
 \EE \left [ \int_0^\rho Q_t^{q-2} d [Q]_t \right ] \geq  \EE \left [ \int_0^\rho S_t^{q-2} d [S]_t \right ]
\end{equation}
with equality only if $\gamma(A \Delta H(A))=0$. \\

By slight abuse of notation, we also define a function $q:[0,1] \to \RR$ by
$$
q(s) = - \int_{-\infty}^{\Phi^{-1}(s)} x d \gamma^1(x)
$$
(where $\Phi(\cdot)$ is the standard Gaussian cumulative distribution function) so that $q(\gamma(B)) = q(H(B))$
for all measurable $B \subset \RR^n$.

Clearly, we will have $Q_0 = \gamma(H(A)) = S_0$. Moreover, observe that if $B$ is a half-space then $q(B) = q(\gamma(B))$. Therefore by repeating the same calculation in which equations (\ref{defst2}) and (\ref{quadvar}) were derived, we have for all $0 \leq t < 1$,
$$
Q_t = \gamma \left ( \frac{H(A) - \tilde W_t}{ \sqrt{1-t}} \right )
$$
and
\begin{equation} \label{quadvarQ}
d [Q]_t = (1-t)^{-1} q \left ( \frac{H(A) - \tilde W_t}{ \sqrt{1-t}} \right )^2 dt = (1-t)^{-1} q \left ( Q_t \right )^2 dt.
\end{equation}
~ \\
Recall that $S_t$ and $Q_t$ are martingales. According to the Dambis / Dubins-Schwartz theorem, there exists an enlargement $(\Omega_1', \Sigma_1', P_1')$ of the probability space $(\Omega_1, \Sigma_1, P_1)$ which supports a process $B(t)$ such that $B(t) - S_0$ is a standard Brownian motion and such that
$$
S_t = B( [S]_t ), ~~ \forall 0 \leq t \leq 1.
$$
In the same manner, we deduce that there exists an enlargement $(\Omega_2', \Sigma_2', P_2')$ of the probability space $(\Omega_2, \Sigma_2, P_2)$ supporting another process $\tilde B(t)$ such that $\tilde B(t) - Q_0$ is a standard Brownian motion and such that
$$
Q_t = \tilde B( [Q]_t ), ~~ \forall 0 \leq t \leq 1.
$$ 

Denote also $T_f = [S]_1, \tilde T_f = [Q]_1$. Note that since the Gaussian measure has a strictly positive density and since $0 <\gamma(A) < 1$, it follows by definition of the processes $S_t, Q_t$ that almost-surely, $0 < Q_t < 1$ and $0 < S_t < 1$ for all $0 \leq t < 1$. Moreover, we clearly have that $S_1 \in \{0,1\}$ and $Q_1 \in \{0,1\}$. Consequently, we have almost-surely
\begin{equation} \label{eqTf}
T_f = \min \bigl \{t>0; ~ B(t) \in \{0,1\} \bigr \}, ~~ \tilde T_f = \min \bigl \{t>0; ~ \tilde B(t) \in \{0,1\} \bigr \}.
\end{equation}
In particular, $T_f, \tilde T_f$ are stopping times with respect to the filtrations corresponding to $B(T), \tilde B(T)$ respectively.

Since both processes $B(T), \tilde B(T)$ are distributed according to the same law, at this point we would like to claim that we may \emph{couple} between the two so that they are defined on the same probability space in a way that the following equation holds almost surely:
\begin{equation} \label{eqcouple}
B(t) = \tilde B(t), ~~ \forall t \geq 0.
\end{equation}

The following theorem summarizes the exact result that we need, in order to establish the existence of this coupling. Its proof relies on an application of the disintegration theorem which enables the construction of a relative product of the corresponding measure spaces, upon conditioning on the Brownian motions. For the precise definitions and notation used in this theorem, see \cite[Chapter 5]{Furstenberg}.

\begin{theorem}
Let $(\Omega_1, \Sigma_1, P_1), ~ (\Omega_2, \Sigma_2, P_2)$ be two regular probability spaces. Let $B_1: \Omega_1 \times [0, \infty) \to \RR, B_2: \Omega_2 \times [0, \infty) \to \RR$ be two standard Brownian motions over the probability spaces $\Omega_1,\Omega_2$ respectively. There exists a measure space $(\Omega, \Sigma, P)$ and two measurable functions $\pi_1:\Omega \to \Omega_1, \pi_2: \Omega \to \Omega_2$ such that \\
(i) $P_1$ is the push-forward of $P$ under $\pi_1$. \\
(ii) $P_2$ is the push-forward of $P$ under $\pi_2$. \\
(iii) For $P$-almost every $\omega \in \Omega$, one has $B_1(\pi_1(\omega), t) = B_2(\pi_2(\omega), t)$ for all $t \in [0, \infty)$.
\end{theorem}

\begin{proof}
Let $Y$ be the space of continuous functions from $[0, \infty)$ to $\RR$, let $\mathcal{D}$ be the $\sigma$-algebra over $Y$ generated by Brownian events and let $\nu$ be a probability measure over $(Y, \mathcal{D})$ such that a random function with law $\nu$ is a standard Brownian motion. Remark that $B_1, B_2$ can be regarded as homomorphisms $\alpha_1: (\Omega_1, \Sigma_1, P_1) \to (Y, \mathcal{D}, \nu), ~\alpha_2:(\Omega_2, \Sigma_2, P_2) \to (Y, \mathcal{D}, \nu)$ respectively. Also, let $\mathcal{M}(\Omega_1)$ and $\mathcal{M}(\Omega_2)$ denote the spaces of probability measures on $(\Omega_1, \Sigma_1), (\Omega_2, \Sigma_2)$ respectively.

According to the disintegration theorem (see \cite[Theorem 5.8, p. 108]{Furstenberg}), there exist measurable maps from $Y$ to $\mathcal{M}(\Omega_1)$ and $\mathcal{M}(\Omega_2)$, denoted by $y \to \mu_{1,y}$ and $y \to \mu_{2,y}$ respectively, such that \\
(a) For every $f \in L^1(\Omega_1,\Sigma_1, P_1)$, we have for $\nu$-almost every $y \in Y$, $f \in L^1(\Omega_1, \Sigma_1, \mu_{1,y})$ and $\EE_{P_1}[f|B_1=y] = \int f d \mu_{1,y}$, \\
(b) For every $f \in L^1(\Omega_1,\Sigma_2, P_2)$, we have for $\nu$-almost every $y \in Y$, $f \in L^1(\Omega_2, \Sigma_2, \mu_{2,y})$ and $\EE_{P_2}[f|B_2=y] = \int f d \mu_{2,y}$. \\
We can now set $\Omega = \Omega_1 \times \Omega_2$ and $\Sigma = \Sigma_1 \times \Sigma_2$ and define the corresponding measure $P$ on $(\Omega, \Sigma)$ by the equation 
$$
P(A) := \int_Y \bigl (\mu_{1,y} \times \mu_{2,y} (A) \bigr ) d \nu(y)
$$
for all $A \in \Sigma$. Finally, let $\pi_1, \pi_2$ be the projections of $\Omega$ onto its components $\Omega_1$ and $\Omega_2$. Note that (a) and (b) imply that $\pi_1$ and $\pi_2$ are homomorphisms from $(\Omega, \Sigma, P)$ to $(\Omega_1, \Sigma_1, P_1)$ and $(\Omega_2, \Sigma_2, P_2)$ respectively so that parts (i) and (ii) of the statements of the theorem hold. Part (iii) of the theorem now follows directly from \cite[Proposition 5.11, p.112]{Furstenberg} which asserts that $\alpha_1 \circ \pi_1(\omega) = \alpha_2 \circ \pi_2(\omega)$ for $P$-almost-every $\omega \in \Omega$.
\end{proof}


\begin{remark} \label{remarkcoupling}
This theorem is in fact stronger than what we need in this setting. Further scrutiny of the process $Q_t$ reveals that there is a $1$ to $1$ mapping between this process and the process $\tilde B(t)$ stopped at $\tilde T_f$. In the end of this section, we will present an alternative justification of the existence of the coupling \eqref{eqcouple} which is based on this fact.
\end{remark}

From this point on, by applying the above theorem to the probability spaces $(\Omega_1', \Sigma_1', P_1')$ and $(\Omega_2', \Sigma_2', P_2')$ with the corresponding processes $B(t)-S_0$ and $\tilde B(t) - S_0$, we can consider the processes $W_t, \tilde W_t$, $S_t$, $Q_t$, $B(t)$ and $\tilde B(t)$ to be defined on one probability space $(\Omega, \Sigma, P)$ and assume that equation \eqref{eqcouple} is satisfied.

Define the functions $T_1(t) = [S]_t$ and $T_2(t) = [Q]_t$. Remark that, by equations \eqref{quadvar} and \eqref{quadvarQ}, these functions are almost surely continuous and strictly increasing in $0 < t < 1$. We can thus define $\tau_1, \tau_2$ to be their respective inverse functions. Equation (\ref{quadvar}) written differently is
\begin{equation} \label{eq1}
\frac{d}{dt} [S]_t = (1-t)^{-1} q^2(A_t),
\end{equation}
which, under the change of variables $t \to \tau_1(T)$ becomes
$$
T_1'( \tau_1 (T) ) = (1- \tau_1(T) )^{-1} q( A_{\tau_1(T)} )^2, ~~ \forall 0 \leq T < T_f.
$$
Remark that for all $0 < T < T_f$ one has $0 < \tau_1(T) < 1$ and $q(A_{\tau_1(T)}) > 0$. Thus, $\tau_1(T)$ is differentiable for all $0 < T < T_f$ and we can apply the inverse derivative formula to get
\begin{equation} \label{tau1diff}
\tau_1'(T) = (1 - \tau_1(T)) q(A_{\tau_1(T)})^{-2}, ~~ \forall 0 \leq T < T_f.
\end{equation}
Similarly, by (\ref{quadvarQ}) and by the fact that $Q_{\tau_2(T)} = B(T)$,
\begin{equation} \label{tau2diff}
\tau_2'(T) = (1 - \tau_2(T)) q(B(T))^{-2}, ~~ \forall 0 \leq T < T_f.
\end{equation}
Finally, define
$$
\omega_1(T) = - \log(1 - \tau_1(T)), ~~\omega_2(T) = - \log(1 - \tau_2(T)).
$$
So by (\ref{tau1diff}) and (\ref{tau2diff}), we have
\begin{equation} \label{deromega}
\omega_1(T)' = q(A_{\tau_1(T)})^{-2}, ~~ \omega_2(T)' = q(B(T))^{-2}
\end{equation}
for all $0 \leq T < T_f$. \\ \\
The following claim will provide the only inequality in the proof of the theorem. Its proof is very simple, and we postpone it to the end of the section.
\begin{claim} \label{claim1}
For all measurable $B \subset \RR^n$, one has
\begin{equation} 
q(B) \leq q(\gamma(B)) = q(H(B))
\end{equation}
with equality if and only if $\gamma(B \Delta H(B)) = 0$.
\end{claim}
Recall that $\gamma(A_t) = S_t$, so $B(T) = \gamma(A_{\tau_1(T)})$. The above claim implies that
$$
q(A_{\tau_1(T)}) \leq q(B(T)).
$$
Consequently,
\begin{equation} \label{eqineq}
\omega_1'(T) - \omega_2'(T) = q(A_{\tau_1(T)})^{-2} - q(B(T))^{-2} \geq 0 , ~~ \forall 0 \leq T < T_f,
\end{equation}
which implies, by definition, that $\tau_1(T) \geq \tau_2(T)$ for all $1 \leq T < T_f$. In other words
\begin{equation} \label{eqleq}
[S]_t \leq [Q]_t, ~~ \forall 0 \leq t \leq 1.
\end{equation}
By substituting $T = T_2(t)$, so that $d T = T_2'(t) dt = d [Q]_t$, the left hand side of equation (\ref{eq2}) becomes
$$
\EE \left [ \int_0^\rho Q_t^{q-2} d [Q]_t \right ] =  \EE \left [ \int_0^{[Q]_\rho} B(T)^{q-2} d T \right ]
$$
and by substituting $T = T_1(t)$, the right hand side becomes
$$
\EE \left [ \int_0^\rho S_t^{q-2} d [S]_t \right ] =  \EE \left [ \int_0^{[S]_\rho}  B(T)^{q-2} d T \right].
$$
Equation (\ref{eq2}) shows us that our goal is to prove that
\begin{equation} \label{last}
\EE \left [ \int_0^{[Q]_\rho} B(T)^{q-2} d T \right ] \geq \EE \left [ \int_0^{[S]_\rho} B(T)^{q-2} d T \right].
\end{equation}
Since both integrands are positive and in light of (\ref{eqleq}), the proof of the inequality complete. \\

Let us now turn to analyse the equality case. We first remark that, by equation \eqref{eqTf}, we have almost surely $B(T) > 0$ for all $T \leq \max \{[S]_\rho, [Q]_\rho\}$. It follows that there could only be an equality in formula (\ref{last}) if $[S]_\rho = [Q]_\rho$, which implies that for almost all $0 \leq t \leq \rho$, there is an equality in equation (\ref{eqineq}). But according to Claim \ref{claim1}, the only case in which there can be equality in equation (\ref{eqineq}) is if $\gamma(A_t \Delta H(A_t)) = 0$ which in turn implies that $\gamma(A \Delta H(A)) = 0$. The equality case is thus also established. \\

\begin{remark}
Note that by equations \eqref{eqTf} and \eqref{eqcouple}, we have that $[S]_1 = [Q]_1$ almost surely. This could cause a confusion considering the fact that if $\gamma(A \Delta H(A)) > 0$ then almost surely, $[S]_t < [Q]_t$ for all $0 < t < 1$ (which we have just established) and in light of equation \eqref{eqineq}. This confusion is however settled by the fact that $t \to \omega_1(T_1(t))$ is not continuous at $t=1$.
\end{remark}
\medskip
It remains to prove Claim \ref{claim1}. It will be useful to first prove the following more general fact:

\begin{lemma} \label{claim1alt}
Let $m:\RR \to [0,1]$ be a measurable function. One has,
$$
\left | \int_{\RR} x m(x) d \gamma^1(x) \right | \leq q \left ( \int_{\RR} m(x) d \gamma^1(x) \right )
$$
and there is an equality in the above if an only if there exists $\alpha \in \RR$ such that $m(x)$ is of either the form $\mathbf{1}_{x \leq \alpha}$ or of the form $\mathbf{1}_{x \geq \alpha}$ for almost every $x \in \RR$.
\end{lemma}
\begin{proof}
Without loss of generality, we may assume that $\int_{\RR} x m(x) d \gamma^1(x) \geq 0$ (otherwise replace $m(x)$ by $m(-x)$). Let $h(x)$ be the function of the form $h(x) = \mathbf{1}_{x \geq \alpha}$ where $\alpha$ is chosen such that $\int h(x) d \gamma^1(x) = \int m(x) d \gamma^1(x)$. The claim of the lemma boils down to showing that
\begin{equation}
\int_{\RR} x (h(x) - m(x)) d \gamma^1(x) \geq 0
\end{equation}
with equality if and only if $m(x) = h(x)$ almost surely. Indeed, we have
$$
\int_{\RR} x (h(x) - m(x)) d \gamma^1(x) = \int_{\RR} (x - \alpha) (h(x) - m(x)) d \gamma^1(x)
$$
now, by definition of $h(x)$ and by the fact that $0 \leq m(x) \leq 1$ for all $x$, the function $h(x) - m(x)$ has the same sign as $(x-\alpha)$, which means that the right hand side of the above equation is non-negative, and is zero if and only if $h(x)=m(x)$ almost surely.
\end{proof}

\begin{proof}[\textbf{Proof of Claim \ref{claim1}}]
Let $B \subset \RR^n$. Define,
$$
\theta = \frac{\int_{B} x \gamma(x) dx}{|\int_{B} x \gamma(x) dx|}.
$$
(if the denominator is zero then there's nothing to prove). Let $\mu$ be the push-forward of the restriction of $\gamma$ to $B$ under the map $x \to \langle \theta, x \rangle$. Define $m(x) = \frac{d \mu}{d \gamma^1 }(x)$. Since $\gamma(x)$ is a product measure,
$$
q(B) = \left | \int_{B} \langle \theta, x \rangle \gamma(x) dx  \right | =
\left | \int_{\RR} x m(x) \gamma^1(x) dx \right |.
$$
An application of lemma (\ref{claim1alt}) finishes the proof.
\end{proof}

We conclude with a remark about a possible generalization of this approach.
\begin{remark}
Let $\phi:[0,\infty) \to \RR$ be a twice-differentiable strictly convex function. One may generalize the definition of the $q$-stability of a set $A \subset \RR^n$ by defining
$$
\St^\phi_\rho(A) = \EE \Bigl [ \phi \bigl (\PP( \sqrt{\rho} X + \sqrt{1-\rho} Y \in A ~ | X ) \bigr )  \Bigr ]
$$
where $X,Y$ are independent standard Gaussian random vectors. It is not hard to check that the same proof, with a slight modification, would still work if the expression $\St^q_\rho(A)$ is replaced by $\St^\phi_\rho(A)$. One particularly interesting special case would be $\phi(x) = x \log x$ which reproves the fact that among all measurable sets $A \subset \RR^n$ of a given Gaussian measure, the relative entropy of $P_\rho(A)$ with respect to the Gaussian measure is maximized for half-spaces (where $P_\rho$ is the Ornstein-Uhlenbeck operator defined in equation \eqref{defOU} below).
\end{remark}

\subsection{An alternative construction of the coupling}
As promised in Remark \ref{remarkcoupling}, for the reader's benefit, we sketch an alternative justification of the existence of a coupling satisfying equation \eqref{eqcouple} which may shed some light on this construction. The idea is, instead of constructing the process $Q_t$ via equation \eqref{defQt}, to construct it directly from the Brownian motion $B(T)$ in a way that $\tilde B(T) = B(T)$. 

Let $B(T)$ be a process such that $B(0) - S_0$ is a standard Brownian motion and let $T_f := \min \{T; B(T) \in \{0,1\} \}$. Consider the ordinary differential equation
$$
\tau_2(0) = 0, ~~ \tau_2'(T) = (1-\tau_2(T)) q(B(T))^{-2}.
$$
Note that the function $F(T,y) = (1-y) q(B(T))^{-2}$ is Lipschitz with respect to $y$ and continuous with respect to $T$ whenever $0 < T < T_f$. Thus, existence and uniqueness of a solution to this equation in the interval $0 \leq T < T_f$ follows from the Picard-Lindel\"{o}f theorem.

Remark that by definition one has $\tau_2'(t) > 0$ for all $0 < T < T_f$, thus the inverse function theorem allows us to define $T_2(t)$ as the inverse function of $\tau_2$, and conclude that
$$
T_2'(t) = (1-t)^{-1} q(B(T_2(t)))^2, ~~ \forall 0 < t < \tau_2(T_f).
$$
Finally, define
$$
Q_t = B(T_2(t)).
$$
Consider the filtration $\mathcal{F}'_t := \sigma(T_2(t), \{ B(u) \}_{0 < u < T_2(t)} )$. We claim that $Q_{t \wedge \tau_2(T_f)}$ is a martingale with respect this filtration. Indeed, let $s,t$ be two $\FF'_t$-stopping times such that $0 \leq t \leq s \leq \tau_2(T_f)$ almost surely. Observe that $T_2(s) = \min \{T; \tau_2(T) = s \}$ is a stopping time with respect to the filtration of $B(T)$. By the optional stopping theorem, we have
$$
\EE[Q_s | \mathcal{F}'_t] = \EE[B( T_2(s) ) | \mathcal{F}'_t  ] 
$$
$$
= \EE[B( T_2(s) ) | \{B_u\}_{0 \leq u \leq T_2(t)} ] = B(T_2(t)) = Q_t.
$$
This shows that $Q_t$ is a local-martingale, and since $0 \leq Q_t \leq 1$ for all $0\leq t \leq \tau_2(T_f)$ we deduce that $Q_{t\wedge \tau_2(T_f)}$ is a martingale. Moreover, we have $[Q]_t = T_2(t)$ and so
$$
d [Q]_t = T_2(t)' dt = (1-t)^{-1} q(Q_t)^2 dt.
$$
By the martingale representation theorem, we conclude that $Q_t$ must satisfy the equation
$$
d Q_t = (1-t)^{-1/2} q(Q_t) d \bar B(t), ~~ Q_0 = S_0
$$
where $\bar B(t)$ is a Brownian motion. By the uniqueness of the solution of this stochastic differential equation, we conclude that the process $Q_t$ has the same law as the one of the process defined by equation \eqref{defQt} and one must also have $\lim_{T \to T_f^-} \tau_2(T) = 1$.

\section{The robustness estimate}

The goal of this section is to prove Theorem \ref{thmrobust} and Corollary \ref{cortv}. \\ 

Let us briefly describe the main steps of our proof. Our starting point is equations (\ref{needtoshow}) and (\ref{needtoshow2}),
according to which we have
$$
\St^q_\rho(H(A)) - \St^q_\rho(A) = \EE[Q_\rho^q] - \EE[S_\rho^q] =
$$
$$
\frac{1}{2} q(q-1) \left (\EE \left [ \int_0^\rho Q_t^{q-2} d [Q]_t \right ] - \EE \left [ \int_0^\rho S_t^{q-2} d [S]_t \right ] \right ).
$$
As described above, we couple the processes $S_t$ and $Q_t$ using the equations
$$
S_t = B([S]_t), ~~Q_t = B([Q]_t)
$$
where $B(T) - S_o$ is a standard Brownian motion and $T_f = [S]_1 = [Q]_1$. The above equation becomes
\begin{equation} \label{mainthing}
\St^q_\rho(H(A)) - \St^q_\rho(A) = \frac{1}{2} q(q-1) \EE \left [ \int_{0}^{[Q]_\rho} B(T)^{q-2} d T - \int_{0}^{[S]_\rho} B(T)^{q-2} d T \right ] =
\end{equation}
$$
\frac{1}{2} q(q-1) \EE \left [ \int_{[S]_\rho}^{[Q]_\rho} B(T)^{q-2} d T \right ].
$$
As in the previous section, we define
$$
T_1(t) = [S]_t, ~~ T_2(t) = [Q]_t
$$
and denote by $\tau_1, \tau_2$ their corresponding inverse functions. We also define
$$
\omega_1(T) = - \log(1 - \tau_1(T)), ~~\omega_2(T) = - \log(1 - \tau_2(T)).
$$
We have, as in (\ref{deromega}),
\begin{equation} \label{2deromega}
\omega_1(T)' = q(A_{\tau_1(T)})^{-2}, ~~ \omega_2(T)' = q(B(T))^{-2}.
\end{equation}
Also, by Claim \ref{claim1},
\begin{equation} \label{2eqineq}
\omega_1'(T) - \omega_2'(T) = q(A_{\tau_1(T)})^{-2} - q(B(T))^{-2} \geq 0 , ~~ \forall 0 \leq T < T_f.
\end{equation}
Finally, it will also be convenient to define the stopping times
$$
\Theta_1 = \min \{T ; ~ \omega_1(T) = -  \log(1 - \rho)   \} = T_1(\rho), ~~\Theta_2 = \min \{T ; ~ \omega_2(T) = - \log(1 - \rho)   \} = T_2(\rho).
$$
So equation (\ref{mainthing}) becomes
\begin{equation} \label{mainthing2}
\St^q_\rho(H(A)) - \St^q_\rho(A) = \frac{1}{2} q(q-1) \EE \left [ \int_{\Theta_1}^{\Theta_2} B(T)^{q-2} d T \right ].
\end{equation}

Our goal is to show that the quantity on the right hand side is not too small. For that, we would like to show two things: (i) That the expectation of $\Theta_2 - \Theta_1$ is quite large and (ii) that $B(T)$ is not too close to zero when we reach $\Theta_1$ and thus the integrand will be non-negligible in the (rather large) interval $[\Theta_1, \Theta_2]$.

We will first roughly show that up to time $\Theta_2$, the process $B(T)$ is bounded away from zero and from one with a probability that only depends on $B(0)$ and on $\rho$ (this is done partly in Lemma \ref{LemStoping} and partly in Lemma \ref{lemrho} below). This ensures that the integrand in the above formula is not too small, and hence it will be enough to prove (i).

The main step in the proof of (i) will be to show that
\begin{equation} \label{mainstep}
\PP( \omega_1(T_0) - \omega_2(T_0) \geq \delta) \geq p
\end{equation}
with $\delta$ and $p$ being as large as possible and $T_0 \leq \Theta_1$. This roughly means that the process $Q_t$ is "lagged" with respect to the process $S_t$ (when considering the above coupling) so that in the future, when the process $S_t$ stops (i.e., when $T=\Theta_1$), the process $Q_t$ will still have some time left (until $T=\Theta_2$) in order to accumulate a non-negligible quantity of quadratic variation. In other words, in order to use this fact to control the difference $\Theta_2 - \Theta_1$ from below, we use the fact that $\int_{\Theta_1}^{\Theta_2} \omega_2'(T) d T \geq \delta$, and invoke (\ref{2deromega}) in order to get an upper bound for the expression $\omega_2'(T)$. This is done in Lemma \ref{lemeps} below. \\

In order to prove an equation of the form (\ref{mainstep}), we will define
\begin{equation} \label{defeps}
\epsilon_t = q^2(S_t) - q^2(A_t).
\end{equation}
Note that $\eps_0 = \eps(A)$. We will use formula (\ref{2deromega}), which tells us that
\begin{equation} \label{eqqq}
\omega_1'(T) - \omega_2'(T) = q^{-2} (A_{\tau_1(T)}) - q^{-2 }(S_{\tau_1(T)}) \geq c \epsilon_{\tau_1(T)}
\end{equation}
hence difference $\omega_1'(T) - \omega_2'(T)$ is controlled by the quantity $\epsilon_{\tau_1(T)}$. Thanks to this, in order to prove that the difference $\omega_1(\Theta_1) - \omega_2(\Theta_1)$ is quite large, it will be enough to prove that with a non-negligible probability, one has
\begin{equation} \label{estrip}
\eps_t > c \eps_0, ~~ \forall 0 < t < \alpha
\end{equation}
where $c, \alpha$ are not too small. If this is true, we can integrate equation (\ref{eqqq}) and deduce that for all $t \geq \alpha$ one has $\omega_1(\alpha) - \omega_2(\alpha) \geq c \eps_0 \alpha$. Once we have this, we can finally ensure that $\omega_1(\Theta_1) - \omega_2(\Theta_1) = \delta$ where $\delta$ is not too small, with a non-negligible probability. This is eventually done in Lemma \ref{lembeta} below.

The only fact we will still have to explain is why an estimate of the form (\ref{estrip}) holds (which will be proven in Lemma \ref{epsstrip} below). This is the most involved step of the proof and the two consequent subsections are dedicated to it. The idea of its proof is to calculate the It\^{o} differential of the process $\epsilon_t$ (which turns out to be an It\^{o} process) using formula (\ref{stochastic}) and then bound it in terms of $S_t$ and $\epsilon_t$ itself. An entire subsection is dedicated to the calculation of this differential, and another subsection is dedicated to bounding its drift and quadratic variation, which boils down to bounding the Hilbert-Schmidt norm of a certain matrix. The outcome of these two subsections is concluded in Proposition \ref{mainsec4} below. Finally, the upper bound for the deficit is proven in subsection \ref{subupper}. \\ \\

We are finally ready to begin the proof. We start by defining a stopping time,
$$
\TTT = \min \{T ; ~ B(T)(1-B(T)) \leq B(0)(1-B(0)) / 2 \}.
$$
The following simple lemma shows that we can expect the quantity $S_t(1-S_t)$ to remain bounded away from zero with a non-negligible probability.
\begin{lemma} \label{LemStoping}
There exists a universal constant $c_1>0$ such that
\begin{equation} \label{eqLemStopping}
\PP \left (\tau_1(\TTT) \geq  \frac{1}{2} \right  ) \geq c_1 S_0 (1-S_0).
\end{equation}
\end{lemma}
\begin{proof}
First, it will be useful to notice that, according to (\ref{quadvarQ}),
$$
T_2'(t) = (1-t)^{-1} q(Q_t)^2 \leq (1-t)^{-1} q(1/2)^2 = \frac{1}{2 \pi} (1-t)^{-1}.
$$
Therefore, using (\ref{eqleq}),
\begin{equation} \label{Tbound}
T_1(t) \leq T_2(t) \leq - \frac{1}{2 \pi} \log(1 - t).
\end{equation}
Define
$$
U = \min \{T; ~ B(T)(1-B(T)) \leq B(0)(1-B(0)) / 2 \mbox { or } B(T) = 1/2  \}.
$$
Assume for now that $B(0) \leq 1/2$. Let $\beta$ be the solution to the equation
$$
\beta(1 - \beta) = B(0)(1 - B(0)) / 2
$$
satisfying $\beta < 1/2$. It is easy to verify that
$$
\beta \leq B(0) / 2.
$$
Since $B(T)$ is a martingale, the optional stopping theorem implies that
$$
\PP( B(U) = 1/2 ) = \frac{B(0) - \beta}{1/2 - \beta} \geq B(0).
$$
In a completely similar manner, when $B(0) > 1/2$ one has $\PP( B(U) = 1/2 ) \geq 1-B(0)$, and we conclude that
\begin{equation} \label{pbu}
\PP( B(U) = 1/2 ) \geq B(0)(1 - B(0)).
\end{equation}
Equation (\ref{Tbound}) teaches us that $T_1(1/2) \leq \frac{1}{2 \pi} \log 2$. Clearly, there exists a constant $c>0$ such that a Brownian motion starting at $1/2$ at time $T_0$ remains inside the interval $[1/4,3/4]$ by time $ T_0 + \frac{1}{2 \pi} \log 2$ with probability at least $c$. So, by remarking that
$$
B(T)(1-B(T)) \leq B(0)(1-B(0)) / 2 \Rightarrow B(T)(1-B(T)) \leq 1/8 \Rightarrow B(T) \notin [1/4,3/4],
$$
we learn that
$$
\left . \PP \left ( B(T)(1-B(T) > B(0)(1-B(0)) / 2, ~~ \forall U \leq T \leq U + \frac{1}{2 \pi} \log 2~ \right  | B(U) = 1/2 \right ) > c.
$$
Combining this fact with (\ref{pbu}) gives
$$
\PP(\TTT \geq T_1(1/2)) \geq c B(0) (1 - B(0)) = c S_0 (1 - S_0).
$$
The proof is complete.
\end{proof}

For a number $\delta > 0$, define the event
$$
F_\delta = \left  \{ |\epsilon_t - \epsilon_0| \leq \epsilon_0 / 2, ~~ \forall 0 \leq t \leq \delta |\log \epsilon_0|^{-1} S_0^7(1-S_0)^7 \right  \}
$$
where $\epsilon_t$ is defined in equation (\ref{defeps}).

\begin{lemma} \label{epsstrip}
There exists a universal constant $c_2>0$ such that whenever $\epsilon_0 < 1/2$,
\begin{equation} \label{eqlemeps}
\PP \left (\tau_1(\TTT) \geq  \frac{1}{2} \mbox { and } F_{c_2} \mbox{ holds} \right ) \geq c_2 S_0 (1 - S_0).
\end{equation}
\end{lemma}

The main ingredient of this lemma will be the following proposition, to the proof of which we dedicate the next two subsections. The point of the proposition is that $\epsilon_t$ does not move too much provided that it is small and that $S_t$ is bounded away from $0$ and $1$.

\begin{proposition} \label{mainsec4}
There exists a universal constant $C>0$ such that the following holds: There exist two predictable processes $\alpha_t \in \RR^n$ and $\beta_t \in \RR$ satisfying
$$
d \epsilon_t = \langle \alpha_t, d W_t \rangle + \beta_t dt
$$
and such that the following bounds hold,\\
(i) For all $0 \leq t \leq 1$,
$$
|\alpha_t| \leq (1-t)^{-1/2} \frac{C}{S_t^3(1 - S_t)^3} \epsilon_t \sqrt{|\log \epsilon_t|}.
$$
(ii) For all $0 \leq t \leq 1$,
$$
|\beta_t| \leq (1-t)^{-1} \frac{C}{S_t^3(1-S_t)^3} \epsilon_t \sqrt{|\log \epsilon_t|}.
$$
\end{proposition}
\bigskip
\begin{proof}[\textbf{Proof of Lemma \ref{epsstrip}}]
Define the stopping time
$$
u = \min \{t; ~ |\epsilon_t - \epsilon_0| \geq \epsilon_0 / 2 \} \wedge 1.
$$
By the notation of Proposition \ref{mainsec4} we have
$$
d \epsilon_t = \langle \alpha_t, d W_t \rangle + \beta_t dt.
$$
By definition for all $t \leq \tau_1(\TTT)$, one has
\begin{equation} \label{eq101}
\frac{1}{S_t(1-S_t)} \leq \frac{2}{S_0(1-S_0)}.
\end{equation}
Consequently, according to part (i) of the above proposition, we have
\begin{equation} \label{alphaleq}
|\alpha_t| \leq \frac{C_0} {S_0^3(1-S_0)^3} \epsilon_0 \sqrt{|\log \epsilon_0|}, ~~ \forall 0 \leq t \leq \tau_1(\TTT) \wedge u \wedge 1/2
\end{equation}
for a universal constant $C_0 > 0$ and according to part (ii) of the proposition,
\begin{equation} \label{betaleq}
|\beta_t| \leq \frac{C_1}{S_0^3(1-S_0)^3} \epsilon_0 \sqrt{|\log \epsilon_0|},  ~~ \forall 0 \leq t \leq \tau_1(\TTT) \wedge u \wedge 1/2
\end{equation}
for a universal constant $C_1 > 0$. Fix a constant $\delta > 0$ and define
$$
t_0 = \delta |\log \epsilon_0|^{-1} S_0^7(1-S_0)^7 \wedge \tau_1(\TTT) \wedge u \wedge \frac{1}{2}.
$$
Equation (\ref{betaleq}) and the fact that $\epsilon_0 < 1/2$ imply
\begin{equation} \label{bdbeta}
\left |\int_0^{t_0} \beta_t dt \right | \leq C_3 \delta \epsilon_0
\end{equation}
for some universal constant $C_3>0$. Using the triangle inequality gives
$$
\PP \left ( |\epsilon_{t_0} - \epsilon_0| \geq \epsilon_0 / 2 \right ) \leq
$$
$$
\PP \left ( \left |\int_0^{t_0} \langle \alpha_t, d W_t \rangle  \right | + \left |\int_0^{t_0} \beta_t dt \right | \geq \epsilon_0 / 2 \right ) \leq
$$
$$
\PP \left ( \left |\int_0^{t_0} \langle \alpha_t, d W_t \rangle  \right | \geq \epsilon_0 / 2 - C_3 \delta \epsilon_0 \right ).
$$
By assuming that $\delta$ is a small enough universal constant, we can assert that $\epsilon_0 / 2 - C_3 \delta \epsilon_0 \geq \epsilon_0 / 4$, and obtain
\begin{equation} \label{almost11}
\PP \left ( |\epsilon_{t_0} - \epsilon_0| \geq \epsilon_0 / 2 \right ) \leq \PP \left ( \left |\int_0^{t_0} \langle \alpha_t, d W_t \rangle  \right | \geq \epsilon_0 / 4 \right ).
\end{equation}
To estimate the right hand side, we use equation (\ref{alphaleq}) to get
\begin{equation}
[\epsilon]_{t_0} = \int_0^{t_0} |\alpha_t|^2 dt  \leq C_0 \delta S_0(1-S_0) \epsilon_0^2.
\end{equation}
By It\^{o}'s formula, we have
$$
\EE \left [ \left ( \int_0^{t_0} \langle \alpha_t, d W_t \rangle \right )^2 \right ] \leq C_0 S_0(1-S_0) \delta \epsilon_0^2
$$
and, by Chebyshev's inequality,
$$
\PP \left ( \left |\int_0^{t_0} \langle \alpha_t, d W_t \rangle  \right | > \epsilon_0 / 4 \right ) < 16 C_0 S_0(1-S_0) \delta.
$$
Combining this with (\ref{almost11}) finally gives
$$
\PP \left ( |\epsilon_{t_0} - \epsilon_0| \geq \epsilon_0 / 2 \right ) \leq 16 C_0 S_0(1-S_0) \delta.
$$
This shows that there exists a universal constant $c_2 > 0$ such that if $\delta \leq c_2$, then
$$
\PP \left ( |\epsilon_{t_0} - \epsilon_0| \geq \epsilon_0 / 2 \right ) < c_1 S_0 (1 - S_0) / 2
$$
where $c_1$ is the constant from equation (\ref{eqLemStopping}). In other words, by definition of $t_0$ and $u$ and by the continuity of $\epsilon_t$,
\begin{equation} \label{albound}
\PP(t_0 = u) \leq c_1 S_0 (1 - S_0) / 2.
\end{equation}
Define
$$
\alpha = c_2 |\log \epsilon_0|^{-1} S_0^7(1-S_0)^7.
$$
The assumption that $\epsilon_0 < 1/2$ can ensure (by taking $c_2$ to be small enough) that $\alpha \wedge 1/2 = \alpha$,
which also implies that
$$
t_0 < u \Rightarrow t_0 = \alpha \wedge \tau_1(\TTT).
$$
By definition, if $t_0 = \tau_1(\TTT)$ it means that $\tau_1(\TTT) \leq \frac{1}{2}$ so equation (\ref{albound}) becomes
$$
\PP( |\epsilon_{t} - \epsilon_0| \leq \epsilon_0 / 2, ~~ \forall 0 \leq t \leq \alpha \mbox { or } \tau_1(\TTT) < 1/2 ) > 1 - c_1 S_0 (1 - S_0) / 2.
$$
Using a union bound with the result of Lemma \ref{LemStoping} finishes the proof.
\end{proof}

From this point on, we denote $\alpha = c_2 |\log \epsilon_0|^{-1} S_0^7(1-S_0)^7$ where $c_2$ is the constant which appears in equation (\ref{eqlemeps}) and define
\begin{equation} \label{defg}
G = F_{c_2} \cap  \left \{ \tau_1(\TTT) \geq  \frac{1}{2} \right  \} =
\end{equation}
$$
\left  \{ |\epsilon_t - \epsilon_0| \leq \epsilon_0 / 2, ~~ \forall 0 \leq t \leq \alpha \right \} \cap \left  \{ \tau_1(\TTT) \geq  \frac{1}{2} \right  \}.
$$
According to the previous lemma, we have
$$
\PP(G) \geq c_2 S_0(1 - S_0).
$$
Next, we show:

\begin{lemma} \label{lembeta}
We have, almost surely,
\begin{equation}
G \mbox { holds } \Rightarrow \omega_1(T_1(\alpha)) - \omega_2(T_1(\alpha)) \geq \alpha \epsilon_0 /2.
\end{equation}
\end{lemma}
\begin{proof}
We start with recalling formula (\ref{2eqineq}). According to this formula, we have
\begin{equation} \label{deromega2}
(\omega_1 - \omega_2)'(T) =   q(A_{\tau_1(T)})^{-2} - q(S_{\tau_1(T)})^{-2}.
\end{equation}
Moreover, according to formula (\ref{quadvar})
$$
T_1'(t) = \frac{d}{dt} [S]_t = (1-t)^{-1} q(A_t)^2.
$$
By the chain rule, we get
\begin{equation}
\frac{d}{dt} (\omega_1 - \omega_2)(T_1(t)) = (1-t)^{-1} \left (q(A_t)^{-2} - q(S_t)^{-2} \right ) q(A_t)^2 =
\end{equation}
$$
(1-t)^{-1} \left ( 1 - \frac{q(A_t)^2}{q(S_t)^2} \right ).
$$
Next, we observe that the function $q(\cdot)$ is bounded from above by $q(1/2) < 1$. It follows that for all $0 < t < 1$,
$$
\epsilon_{t} = q(S_{t})^{2} - q(A_{t})^{2} =  q(S_{t})^{2} \left (1 - \frac{q(A_{t})^{2}}{q(S_{t})^{2}} \right ) \leq 1 - \frac{q(A_{t})^{2}}{q(S_{t})^{2}}.
$$
The two last equations yield
$$
\frac{d}{dt} (\omega_1(T_1(t)) - \omega_2(T_1(t))) \geq \epsilon_t, ~~ \forall 0 \leq t \leq 1.
$$
Under the assumption that $G$ holds, by integrating both sides, we get
$$
\omega_1(T_1(\alpha)) - \omega_2(T_1(\alpha)) \geq \int_0^\alpha \epsilon_t dt \geq \alpha \epsilon_0 / 2.
$$
The lemma is complete.
\end{proof}

The next lemma helps us take advantage of the deficit $\omega_1(T) - \omega_2(T)$ in order to give a lower bound for the right hand side of equation (\ref{mainthing2}).

\begin{lemma} \label{lemeps}
Let $0 < \EPS < 1$ and $0 \leq t_0 \leq \rho$. Consider the event
$$
H = \Bigl \{ \omega_1(T_1(t_0)) - \omega_2(T_1(t_0)) \geq \EPS \Bigr \}.
$$
Then almost surely, whenever $H$ holds, one has 
\begin{equation} \label{eqeps1}
\EE \left . \left [ \int_{\Theta_1}^{\Theta_2} B(T)^{q-2} d T \right | \mathcal{F}_{t_0} \right  ] \geq \frac{c \EPS}{2^q} \EE \left . \left [Q_{t_1}^{q+1} (1 - Q_{t_1})^{q+1} \right  | \mathcal{F}_{t_0} \right  ]
\end{equation}
where $t_1$ is defined by the equation
\begin{equation} \label{deft2}
- \log(1 - t_1) = - \log(1 - \rho) - \delta
\end{equation}
and $c$ is a positive universal constant.
\end{lemma}

Before we prove the lemma, we will need the estimate
\begin{lemma} \label{lemestq}
There exist universal constants $c,C>0$ such that for all $0 < s < 1$,
\begin{equation} \label{estq1}
c s (1-s) \leq q(s) \leq C s(1-s) \sqrt{| \log(s(1-s)) |}
\end{equation}
Moreover, the function $q(s)/s$ is decreasing and one has
\begin{equation} \label{estq2}
q(s) \leq C s \sqrt{|\log s|}, ~~ \forall 0 < s < 1
\end{equation}
\end{lemma}
The elementary yet technical proof of this lemma is postponed to the appendix. \\

\begin{proof}[\textbf{Proof of Lemma \ref{lemeps}}]
Define $\tilde \Theta = T_2(t_1)$. By definition, we have
\begin{equation} \label{eqtt1}
\omega_2(\tilde \Theta) = \omega_2(T_2(t_1)) = - \log ( 1 - t_1) = - \log(1 - \rho) - \delta.
\end{equation}
Thanks to equation (\ref{2eqineq}), whenever the event $H$ holds we know that for all $T_1(t_0) \leq T < T_f$, one has
$$
\omega_2(T) \leq \omega_1(T) - \EPS.
$$
In particular, since $t_0 < \rho$, we may take $T = T_1(\rho) = \Theta_1$ in the previous equation, which
gives
\begin{equation} \label{eqtt2}
\omega_2(\Theta_1) \leq \omega_1(\Theta_1) - \EPS = - \log(1 - \rho) - \EPS.
\end{equation}
We conclude from equations (\ref{eqtt1}) and (\ref{eqtt2}) that
\begin{equation} \label{eqtt3}
\Theta_1 \leq \tilde \Theta \leq \Theta_2.
\end{equation}
Moreover, since $\omega_2(\Theta_2) = - \log (1 - \rho)$, equation (\ref{eqtt1}) gives
$$
\omega_2(\Theta_2) - \omega_2 ( \tilde \Theta) = \EPS.
$$
This equation written differently is just
$$
\int_{\tilde \Theta}^{\Theta_2} \omega_2'(T) d T = \EPS
$$
and an application of formula (\ref{2deromega}) yields
$$
\int_{\tilde \Theta}^{\Theta_2} q(B(T))^{-2} d T = \EPS.
$$
Consequently,
$$
(\Theta_2 - \tilde \Theta) \max_{\tilde \Theta \leq T \leq \Theta_2} q(B(T))^{-2} \geq \EPS
$$
or in other words,
$$
\Theta_2 - \tilde \Theta  \geq \EPS \min_{\tilde \Theta \leq T \leq \Theta_2} q(B(T))^2.
$$
Since $q(s) < 1$ for all $s \in [0,1]$ and by the assumption $\delta < 1$ we get that
$$
\Theta_2 - \tilde \Theta \geq \EPS \min_{\tilde \Theta \leq T \leq \tilde \Theta + 1} q(B(T))^2.
$$
(here, in case that $\tilde \Theta + 1 > T_f$, we define $\min_{\tilde \Theta \leq T \leq \tilde \Theta + 1} q(B(T))^2 = 0$). With the estimate (\ref{estq1}), this formula becomes
$$
\Theta_2 - \tilde \Theta \geq c \EPS \min_{\tilde \Theta \leq T \leq \tilde \Theta + 1} B(T)^2(1-B(T))^2
$$
for a universal constant $c>0$. This implies that for all $q>1$,
$$
\int_{\tilde \Theta}^{\Theta_2} B(T)^{q-2} dT \geq c \EPS \min_{\tilde \Theta \leq T \leq \tilde \Theta + 1} B(T)^ {q}(1-B(T))^{q}.
$$
Now, since the expression in the integral is non-negative and by (\ref{eqtt3}), we may integrate on the larger interval $\Theta_1 < T < \Theta_2$ and finally get
\begin{equation} \label{thetas}
\int_{\Theta_1}^{\Theta_2} B(T)^{q-2} dT \geq c \EPS \min_{\tilde \Theta \leq T \leq \tilde \Theta + 1} B(T)^ {q}(1-B(T))^{q}.
\end{equation}
Next, we would like to bound from below the probability that the right hand side is not too small. Define the stopping time
$$
U = \min \left \{ T \geq \tilde \Theta; ~ B(T) = \frac{1}{2} \mbox { or } B(T)(1-B(T)) = B(\tilde \Theta)(1-B(\tilde \Theta) ) / 2 \right  \}.
$$
Since $B(T)$ is a martingale, in complete analogy with the derivation of equation (\ref{pbu}), we get using an optional stopping argument
\begin{equation}
\PP( B(U) = 1/2 ~ | B(\tilde \Theta)) \geq B(\tilde \Theta) (1 - B(\tilde \Theta)).
\end{equation}
and since there exists a constant $c_1>0$ such that a Brownian motion starting at $1/2$ does exit the interval $[1/4,3/4]$ by time $1$ with probability at least $c_1$, we get
$$
\PP \left . \left ( {B(T)(1-B(T)) > B(\tilde \Theta)(1-B(\tilde \Theta)) / 2,} \atop {\forall U \leq T \leq U + 1~}  \right |  B(U)  = 1/2 \right  ) > c_1.
$$
Combined with the previous inequality this becomes
$$
\PP  \left . \left ( \min_{\tilde \Theta \leq T \leq \tilde \Theta + 1} B(T)^{q}(1-B(T))^{q} \geq \frac{B(\tilde \Theta)^{q}(1-B(\tilde \Theta))^{q}}{2^{q}} \right  | ~ B(\tilde \Theta) \right  ) > c_1 B(\tilde \Theta)(1-B(\tilde \Theta)).
$$
Together with equation (\ref{thetas}) and with the assumption $\EPS < 1$, we get
$$
\PP \left . \left ( \int_{\Theta_1}^{\Theta_2} B(T)^{q-2} dT \geq c \EPS \frac{ B(\tilde \Theta)^{q}(1-B(\tilde \Theta))^{q}}{ 2^{q}}  \right | B(\tilde \Theta) \right ) \geq c_1 B(\tilde \Theta)(1-B(\tilde \Theta)).
$$
Taking expectation over $B(\tilde \Theta)$ gives, for a universal constant $c'>0$,
$$
\EE \left . \left [ \int_{\Theta_1}^{\Theta_2} B(T)^{q-2} dT \right | \mathcal{F}_{t_0} \right ] \geq \frac{c' \EPS}{2^{q}} \left  .  \EE \left [ B(\tilde \Theta)^{q+1}(1-B(\tilde \Theta))^{q+1} \right  | \mathcal{F}_{t_0} \right  ] =
$$
$$
\frac{c' \EPS}{2^{q}} \left .  \EE \left [ Q_{t_1}^{q+1}(1-Q_{t_1})^{q+1} \right | \mathcal{F}_{t_0} \right ].
$$
This proves equation (\ref{eqeps1}) and the proof is complete.
\end{proof}

Before we can finally prove the theorem, the only ingredient we need is a bound the right hand side of formula (\ref{eqeps1}), provided in the next lemma. Roughly speaking, this lemma ensures us that when $\Theta_1$ is reached then $B(T)$ is bounded away from $0$ and from $1$ with a large enough probability.

\begin{lemma} \label{lemrho}
There exists a universal constant $c>0$ such that for any number $q>1$ and for all $0 \leq t_0, t_1 \leq 1$ such that $0 \leq t_0 \leq \min(1/2, t_1)$,
$$
\EE[ Q_{t_1}^{q+1} (1 - Q_{t_1})^{q+1} ~ | \mathcal{F}_{t_0} ] \geq c^{q} S_{t_0}^{q+2} (1 - S_{t_0})^{q+2} \sqrt{1 - t_1}.
$$
\end{lemma}
\begin{proof}
Define a stopping time
$$
u = \min \left  \{t \geq \tau_2(T_1(t_0)); ~~ Q_{t} (1 - Q_{t}) = S_{t_0} (1 - S_{t_0}) / 2 \mbox { or } Q_t = 1/2 \right \}.
$$
Since $Q_t$ is a martingale, and since $Q_{\tau_2(T_1(t_0))} = S_{t_0}$, we can use the optional stopping theorem with a similar argument as the one preceding equation (\ref{pbu}) to get
\begin{equation}
\PP( Q_u = 1/2 ~ | \FF_{t_0}) \geq S_{t_0}(1-S_{t_0}).
\end{equation}
We claim that it is enough to show that if $Q_u = 1/2$ and $u < t_1$, then
\begin{equation} \label{stillneed}
\left . \PP \left  (Q_{t_1} \in [1/4,3/4] ~ \right | \tilde \FF_u \right  ) > c \sqrt{1 - t_1}
\end{equation}
for a universal constant $c>0$, where $\tilde \FF_u$ is the $\sigma$-algebra generated by the Brownian motion $W_t$ stopped at time $\tau_1(T_2(u))$. Indeed, define the event $E = \{Q_u = 1/2 \} \cup \{t_1 < u\}$.
By the above equation and by the definition of $u$ we have, under the assumption that (\ref{stillneed}) holds,
$$
\EE[ Q_{t_1}^{q+1} (1 - Q_{t_1})^{q+1} \mathbf{1}_E | \FF_{t_0} ] \geq
$$
$$
\PP(t_1 < u | \FF_{t_0} ) S_{t_0}^{q+1} (1 - S_{t_0})^{q+1} 2^{-q-1} + \PP( Q_u = 1/2 ~ \& ~ t_1 \geq u | \FF_{t_0} ) (1/4)^{2q + 4} c \sqrt{1 - t_1} \geq
$$
$$
\PP(Q_u = 1/2 | \FF_{t_0})  c_1^{q+1} S_{t_0}^{q+1} (1 - S_{t_0})^{q+1} \sqrt{1 - t_1} \geq c_1^{q+1} S_{t_0}^{q+2} (1 - S_{t_0})^{q+2} \sqrt{1 - t_1}.
$$
for a universal constant $c_1 > 0$, which would finish the proof. It yet remains to prove formula (\ref{stillneed}). \\

In order to get an estimate regarding the distribution of $Q_{t_1}$, we recall the original definition of the process
$Q_t$ in equation (\ref{defQt}):
$$
Q_t = \int_{H(A)}  \gamma_{\tilde W_t, \sqrt{1 - t}}(x) dx
$$
where $\tilde W_t$ is a Brownian motion. According to this equation,
$$
\PP(Q_{t_1} \in [1/4,3/4] ~ | \tilde F_u ) =
$$
$$
\PP \left . \left (\int_{H(A)} \gamma_{\tilde W_{t_1}, \sqrt{1 - t_1}}(x) dx \in [1/4,3/4] ~ \right  | \tilde F_u \right  ).
$$
The above formula clearly does not change if we project both $H(A)$ and $\tilde W_t$ on the direction $v(H(A))$. Therefore, we may assume that $H(A) = [\alpha, \infty)$ for some $\alpha \in \RR$. It is easy
to check that
$$
\left  |\tilde W_{t_1} - \alpha \right | < 0.1 \sqrt{1 - t_1} \Rightarrow \int_{H(A)} \gamma_{\tilde W_{t_1}, \sqrt{1 - t_1}}(x) dx \in [1/4,3/4].
$$
Therefore, it is enough to show that whenever $\tilde W_{u} = \alpha$ and $u < t_1$,
\begin{equation} \label{needneed}
\left . \PP \left ( \left  |\tilde W_{t_1} - \alpha \right | < 0.1 \sqrt{1-t_1}  \right | \tilde W_u \right ) > c \sqrt{1-t_1}
\end{equation}
for a universal constant $c>0$. Noting that the assumption $Q_u = 1/2$ implies that $\tilde W_{u} = \alpha$ and recalling that $t_1 < 1$, we deduce that the above will be implied by
$$
\gamma^1 \bigl ([-0.1 \sqrt{1-t_1}, 0.1 \sqrt{1-t_1}] \bigr ) > c \sqrt{1-t_1}
$$
which clearly holds. The lemma is complete.
\end{proof}

We are finally ready to prove our robustness estimate. The proof is just a combination of the lemmas in this section.

\begin{proof}[\textbf{Proof of Theorem \ref{thmrobust}}]
Define
$$
\alpha = c_2 |\log \epsilon_0|^{-1} S_0^7(1-S_0)^7
$$
and
$$
G = \left  \{ |\epsilon_t - \epsilon_0| \leq \epsilon_0 / 2, ~~ \forall 0 \leq t \leq \alpha \right \} \cap \left  \{ \tau_1(\TTT) \geq  \frac{1}{2} \right  \}
$$
as in equation (\ref{defg}) above. According to lemma  \ref{epsstrip}, we have
\begin{equation} \label{pelarge}
\PP(G) \geq c_2 S_0 (1 - S_0).
\end{equation}
According to Lemma \ref{lembeta}, we know that
$$
G \mbox { holds } \Rightarrow \omega_1(T_1(\alpha)) - \omega_2(T_1(\alpha)) \geq \alpha \epsilon_0 /2.
$$
Together with the legitimate assumption that $c_2 < 1$, it is easy to verify that the assumption $\eps(A) \leq e^{-1/\rho}$ guarantees that $\rho \geq \alpha$. Thus, we can invoke Lemma \ref{lemeps} with $t_0 = \alpha$ and $\EPS = \eps_0 \alpha / 2$ to get
\begin{equation} \label{almost1}
\left . \EE \left [\int_{\Theta_1}^{\Theta_2} B(T)^{q-2} dT \right  | G \right  ] \geq 2^{-q} c_3 \epsilon_0 \alpha \EE \left . \left [Q_{t_1}^{q+1} (1 - Q_{t_1})^{q+1} \right  | G \right  ]
\end{equation}
where $t_1$ is defined in equation (\ref{deft2}) and $c_3 > 0$ is a universal constant. Now, it follows from equation (\ref{2eqineq}) that $\Theta_2 \geq \Theta_1$ almost surely. Using this together with equations (\ref{mainthing2}) and (\ref{pelarge}) gives
\begin{equation} \label{bigeq1}
\St^q_\rho(H(A)) - \St^q_\rho(A) \geq \frac{1}{2} q(q-1) \EE \left . \left [\int_{\Theta_1}^{\Theta_2} B(T)^{q-2} dT \right  | G \right  ] \PP(G) \geq
\end{equation}
$$
c_4 2^{-q} (q-1) \eps_0 \alpha S_0 (1 - S_0) \EE \left . \left [Q_{t_1}^{q+1} (1 - Q_{t_1})^{q+1} \right  | G  \right  ]
$$
where $c_4 > 0$ is a universal constant. Next, by invoking Lemma \ref{lemrho} with $t_0 = \alpha$ (and $t_1$ as we have already defined above), we learn that almost surely
$$
\EE \left . \left [ Q_{t_1}^{q+1} (1 - Q_{t_1})^{q+1} ~ \right  | \FF_{\alpha} \right  ] \geq c_5^{q} S_\alpha^{q+2} (1 - S_\alpha)^{q+2} \sqrt{1 - t_1}
$$
and $c_5 > 0$ is a universal constant. By definition of $\TTT$, together with the legitimate assumption that $\alpha < 1/2$, we have that
$$
G \mbox { holds } \Rightarrow S_{\alpha}(1-S_{\alpha}) \geq S_0(1-S_0) / 2,
$$
so by taking expectation over $G$, the two last equations give
$$
\EE \left . \left [ Q_{t_1}^{q+1} (1 - Q_{t_1})^{q+1} ~ \right  | G \right  ] \geq (c_5/2)^{q} S_0^{q+2} (1 - S_0)^{q+2} \sqrt{1 - t_1}.
$$
Combining the last formula with (\ref{bigeq1}) gives
$$
\St^q_\rho(H(A)) - \St^q_\rho(A) \geq c_6^{q} (q-1) \epsilon_0 \alpha S_0^{q+3} (1 - S_0)^{q+3} \sqrt{1-t_1}
$$
where $c_6 > 0$ is a universal constant. Finally, using the definition of $\alpha$, this gives
$$
\St^q_\rho(H(A)) - \St^q_\rho(A) \geq c_7^q (q-1) \epsilon_0 |\log{\epsilon_0}|^{-1} S_0^{q+11} (1 - S_0)^{q+11} \sqrt{1-t_1} =
$$
$$
c_7^q (q-1) \eps(A) |\log{\eps(A)}|^{-1} \gamma(A)^{q+11} (1 - \gamma(A))^{q+11} \sqrt{1-t_1}.
$$
for a universal constant $c_7 > 0$. The proof of the lower bound is complete. The upper bound is proven in subsection \ref{subupper}.
\end{proof}
\bigskip
Once Theorem \ref{thmrobust} is established, the proof of Corollary (\ref{cortv}) is reduced to a simple upper bound to $\delta(A)^2$ in terms of $\eps(A)$.

\begin{proof} [\textbf{Proof of Corollary \ref{cortv}}]

Suppose that $\gamma(A \Delta H(A)) = \delta$. Observe first that suffices to show that
\begin{equation} \label{suff}
q(H(A)) - q(A) \geq c \delta^2
\end{equation}
for a universal constant $c>0$. Indeed, this assumption combined with the bound (\ref{estq1}) would attain
$$
\eps(A) = q(H(A))^2 - q(A)^2 = (q(H(A)) - q(A))(q(H(A)) + q(A)) \geq
$$
$$
c_1 \gamma(A) (1 - \gamma(A)) \delta^2
$$
and plugging this fact into equations (\ref{robustineq}) and (\ref{robustineq2}) would finish the proof. \\

We turn to prove (\ref{suff}) which is, in some sense, a quantitative version of Claim \ref{claim1}. Define
$$
\theta = \frac{\int_{H(A)} x \gamma(x) dx}{\left |\int_{H(A)} x \gamma(x) dx \right |}.
$$
Denote $\mu = \gamma|_A$, the restriction of the Gaussian measure to $A$, and let $\tilde \mu$ be the push-forward
of $\mu$ under the map $x \to \langle \theta, x \rangle$. We have by definition
\begin{equation}
q(A) = \left | \int_{\RR^n} \langle \theta, x \rangle d \mu (x)  \right | = \left | \int_{\RR} x d \tilde \mu(x) \right |.
\end{equation}
Clearly, the measure $\tilde \mu$ is absolutely continuous with respect to $\gamma^1(x)$, and we may define $m(x) = \frac{d \tilde \mu}{d \gamma^1}(x)$. \\

The choice of the direction of $\theta$ determines that $\int_{\RR} x m(x) d \gamma^1(x) \geq 0$. Let $h(x)$ be the function of the form $h(x) = \mathbf{1}_{x \geq \alpha}$ where $\alpha$ is chosen such that $\int h(x) d \gamma^1(x) = \int m(x) d \gamma^1(x)$. By definition
$$
q(H(A)) = \int_{\RR} x h(x) d \gamma^1(x).
$$
The proof is reduced to showing
\begin{equation} \label{sufff}
\int_{\RR} x (h(x) - m(x)) d \gamma^1(x) \geq c \delta^2.
\end{equation}
We write
$$
\int_{\RR} x (h(x) - m(x)) d \gamma^1(x) = \int_{\RR} (x - \alpha) (h(x) - m(x)) d \gamma^1(x).
$$
Now, by definition of $h(x)$ and by the fact that $0 \leq m(x) \leq 1$ for all $x$, the function $h(x) - m(x)$ has the same sign as $(x-\alpha)$, so it is enough to prove that
\begin{equation} \label{eqgx}
\int_{\RR} |x - \alpha| g(x) d \gamma^1(x) \geq c \delta^2
\end{equation}
where $g(x) =  |h(x) - m(x)| \gamma^1(x)$ and $c>0$ is a universal constant. Thanks to the fact that
$$
\langle x, \theta \rangle \geq 0, ~~ \forall x \in H(A) \setminus A
$$
and
$$
\langle x, \theta \rangle \leq 0, ~~ \forall x \in A \setminus H(A)
$$
we learn that
$$
\int_{\RR} g(x) dx = \gamma(A \Delta H(A)) = \delta.
$$
Finally, since $\gamma^1(x) \leq 1$ for all $x$, we have $g(x) \leq 1$, and by Markov's inequality
$$
\int_{ \{ |x - \alpha| \geq \delta / 4  \}  } g(x) dx \geq \delta / 2.
$$
Since $g$ is non-negative, we get
$$
\int_{\RR} |x - \alpha| g(x) dx \geq \delta^2 / 8
$$
so equation (\ref{eqgx}) is proven and the corollary is established.
\end{proof}

\subsection{Calculation of the differential}

This entire subsection, which is the first step in the proof of Proposition \ref{mainsec4} is dedicated to the calculation of the
differential of the process
$$
\epsilon_t = q(S_t)^2 - q(A_t)^2.
$$

It will be a straight-forward calculation based on repeated use of It\^{o}'s formula and on equation (\ref{stochasticfub}). Before we begin the calculation, we introduce a few definitions and recall some facts from section \ref{sec2}.

Our starting point is formula (\ref{dst}), which reads
$$
d S_t = (1-t)^{-1} \left \langle \int_A (x - W_t) F_t(x) dx, d W_t \right  \rangle.
$$
Define,
$$
V_t = \int_A (x - W_t) F_t(x) dx
$$
so
$$
d S_t = (1-t)^{-1} \langle V_t, d W_t \rangle
$$
and
\begin{equation} \label{vtqv}
d [S]_t = (1-t)^{-2} |V_t|^2 dt.
\end{equation}
It will also be convenient to define the linear map
$$
L_t(x) := \frac{x - W_t}{\sqrt{1-t}}
$$
so that $L_t$ pushes forward the measure whose density is $F_t(x)$ to the standard Gaussian measure. Also note that $A_t = L_t A$. By substituting $y = L_t x$, we have
$$
d [S]_t = (1-t)^{-2} \left | \int_A (x - W_t) F_t(x) dx \right |^2 dt =
$$
$$
(1-t)^{-1} \left | \int_{A_t} y d \gamma(y) \right |^2 dt = (1-t)^{-1} q(A_t)^2 dt
$$
so together with equation (\ref{vtqv}) we have
$$
(1-t)^{-1} q(A_t)^2 = (1-t)^{-2} |V_t|^2.
$$
This encourages us to define
$$
U_t = \frac{V_t}{\sqrt{1-t}}, ~~ u_t = \frac{U_t}{|U_t|}
$$
so that
\begin{equation} \label{eqUt}
U_t = \int_{A_t} x d \gamma(x), ~~ q(A_t) = |U_t|.
\end{equation}
So far, we have established that
\begin{equation} \label{eps22}
\epsilon_t = q(S_t)^2 - |U_t|^2 = q(S_t)^2 - \frac{|V_t|^2}{1-t}.
\end{equation}
$$
~
$$
We are finally ready to begin differentiating, and we start with the second term. We first calculate, using equations \eqref{stochastic} and \eqref{stochasticfub},
$$
d V_t = d \int_A (x - W_t) F_t(x) dx =
$$
$$
- d W_t \int_A  F_t(x) dx + \int_A (x - W_t) d F_t(x) dx - d \left [ W, S \right ]_t =
$$
$$
- d W_t \int_A F_t(x) dx  + (1-t)^{-1} \left ( \int_A (x - W_t)^{\otimes 2} F_t(x) dx \right ) d W_t - (1-t)^{-1} V_t dt =
$$
(substituting $x \to L_t(x)$ in the first and second terms)
$$
- \gamma(A_t) d W_t + \left ( \int_{A_t} x \otimes x d \gamma(x) \right ) d W_t - (1-t)^{-1} V_t dt =
$$
$$
B_t d W_t - (1-t)^{-1} V_t dt
$$
where
\begin{equation} \label{defbt}
B_t = \int_{A_t} \left ( x \otimes x - \mathrm{Id} \right ) \gamma(x) dx.
\end{equation}
Next, we have
$$
d \left ( |V_t|^2 \right ) = \frac{-2 |V_t|^2 dt}{1-t}  + 2 \langle V_t, B_t d W_t \rangle +  \left \Vert B_t \right \Vert_{HS}^2 dt
$$
where the last term stands for the squared Hilbert-Schmidt norm of $B_t$. And so
\begin{equation} \label{dVt2}
d \left ( \frac{|V_t|^2}{1-t} \right ) = - (1-t)^{-2} |V_t|^2 dt  + 2 (1-t)^{-1} \langle B_t V_t, d W_t \rangle + (1-t)^{-1} \left \Vert B_t \right \Vert_{HS}^2 dt.
\end{equation}
In other words,
\begin{equation} \label{eps23}
d |U_t|^2 = - (1-t)^{-1} |U_t|^2 dt +
2 (1-t)^{-1/2} \left \langle B_t u_t, d W_t \right \rangle |U_t| + (1-t)^{-1} \left \Vert B_t \right \Vert_{HS}^2 dt.
\end{equation}
\bigskip
Our next goal is to calculate the differential of the term
$$
q(S_t) = - \frac{1}{\sqrt{2 \pi}} \int_{- \infty}^{\Psi(S_t) } x e^{-x^2/2} dx 
$$
where $\Psi(s) = \Phi^{-1}(s)$ is the inverse Gaussian cumulative distribution function. First, we calculate the derivatives of the function $q(\cdot)$:
\begin{equation} \label{derq1}
q'(s) = - \frac{1}{\sqrt{2 \pi}} \Psi'(s) \Psi(s) \exp(- \Psi(s)^2 / 2) =
\end{equation}
$$
- \frac{1}{\sqrt{2 \pi}} \frac{1}{\Phi'(\Psi(s))} \Psi(s) \exp(- \Psi(s)^2 / 2) =
$$
$$
- e^{\Psi(s)^2/s} e^{-\Psi(s)^2 / 2} \Psi(s) = - \Psi(s).
$$
Also
\begin{equation} \label{derq2}
q''(s) = - \Psi'(s) = - \frac{1}{\Phi'(\Psi(s))} = - \sqrt{2 \pi} e^{\Psi(s)^2/2}.
\end{equation}
Using these derivatives, It\^{o}'s formula now yields
$$
d q(S_t) = q'(S_t) d S_t + \frac{1}{2} q''(S_t) d [S]_t =
$$
$$
- \Psi(S_t) d S_t - (1-t)^{-2} \sqrt{\pi/2} e^{\Psi(S_t)^2/2} |V_t|^2 dt.
$$
Next, we observe the identity,
\begin{equation} \label{ibyparts}
- \Psi(x) = q(x)^{-1} \frac{1}{\sqrt{2 \pi}} \int_{-\infty}^{\Psi(x)} (s^2 - 1) e^{-s^2 / 2} ds.
\end{equation}
Indeed, by integration by parts
$$
\int_{-\infty}^{\Psi(x)} s^2 e^{-s^2 / 2} ds = - \Psi(x) e^{- \Psi(x)^2 / 2} + \int_{-\infty}^{\Psi(x)} e^{-s^2/2} ds =
$$
$$
- \Psi(x) e^{- \Psi(x)^2 / 2} + \sqrt{2 \pi} x,
$$
so
$$
\int_{-\infty}^{\Psi(x)} (s^2 - 1) e^{-s^2 / 2} ds = - \Psi(x) e^{- \Psi(x)^2 / 2}.
$$
Moreover,
\begin{equation} \label{q2}
q(x) = - \frac{1}{\sqrt{2 \pi}} \int_{-\infty}^{\Psi(x)} s e^{-s^2/2} ds =  \frac{1}{\sqrt{2 \pi}} e^{-\Psi(x)^2/2}.
\end{equation}
Combining these two equalities yields (\ref{ibyparts}). Plugging (\ref{ibyparts}) and (\ref{q2}) into the formula for $d q(S_t)$ above gives,
$$
d q(S_t) = q(S_t)^{-1} \left (\frac{1}{\sqrt{2 \pi}} \int_{-\infty}^{\Psi(S_t)} (s^2 - 1) e^{-s^2 / 2} ds \right ) d S_t - \frac{1}{2} (1-t)^{-2} q(S_t)^{-1} |V_t|^2 dt =
$$
$$
q(S_t)^{-1} \left ( \xi(S_t) d S_t - \frac{1}{2} (1-t)^{-2} |V_t|^2 dt \right )
$$
where
$$
\xi(x) = \frac{1}{\sqrt{2 \pi}} \int_{-\infty}^{\Psi(x)} (s^2 - 1) e^{-s^2 / 2} ds.
$$
We continue calculating
$$
d q^2(S_t) = 2 q(S_t) d q(S_t) + d [q(S)]_t =
$$
\begin{equation} \label{dqt2}
2 \xi(S_t) d S_t - (1-t)^{-2} |V_t|^2 dt + (1-t)^{-2} \xi(S_t)^2 \frac{|V_t|^2}{q(S_t)^2} dt.
\end{equation}

The reader may note the similarity between $\xi(S_t)$ and the matrix $B_t$ which encourages us to define
$$
\tilde B_t := \int_{H(A_t)} \left (x \otimes x - \mathrm{Id} \right ) d \gamma(x).
$$
It is straightforward to verify that for all $v \in \RR^n$ one has $\tilde B_t v = \xi(S_t) u_t \langle v, u_t \rangle$, which gives
$$
(1-t)^{-1/2} \left \langle \tilde B_t u_t, d W_t \right \rangle |U_t| = (1-t)^{-1/2} \xi(S_t) \langle U_t, d W_t \rangle = \xi(S_t) d S_t.
$$
Formula (\ref{dqt2}) becomes
$$
d \left (q(S_t)^2 \right ) =
$$
$$
2 (1-t)^{-1/2} \left \langle \tilde B_t u_t, d W_t \right \rangle |U_t| - (1-t)^{-1} |U_t|^2 dt + (1-t)^{-1} ||\tilde B_t||^2_{HS}  \frac{|U_t|^2}{q(S_t)^2} dt.
$$
Combining the last equation with (\ref{eps22}) and (\ref{eps23}) finally gives
\begin{equation} \label{finald}
d \epsilon_t = d \left ( q(S_t)^2 - |U_t|^2 \right ) =
\end{equation}
$$
2 (1-t)^{-1/2} |U_t| \left  \langle \left (\tilde B_t - B_t \right) u_t , d W_t \right \rangle - (1-t)^{-1} \Vert B_t \Vert^2_{HS} dt + (1-t)^{-1} \left \Vert \tilde B_t \right \Vert^2_{HS} \frac{|U_t|^2}{q(S_t)^2} dt.
$$
\subsection{Bounding the differential}
This subsection is dedicated to bounding the right hand side of equation (\ref{finald}) in terms of $\epsilon_t$ and $S_t$, thus proving Proposition \ref{mainsec4}. The proof of this bound will be carried out in three main lemmas, each of which uses a different idea. A glance at formula (\ref{finald}) shows that, in order for the differential of $\epsilon_t$ to be small (in the sense of both drift and quadratic variation) one should show that the matrices $B_t$ and $\tilde B_t$ are quite close to each other in a certain sense.

Recall that $\tilde B_t$ is a rank-one matrix of the form $\alpha u_t \otimes u_t$ for a constant $\alpha \in \RR$. We should therefore expect that the matrix $B$ is close to such a rank-one matrix.

Define $E = \mathrm{sp} \{u_t \}$ and let $P_E, P_{E^\perp}$ be the orthogonal projections onto $E$ and $E^{\perp}$ respectively. Our first lemma (Lemma \ref{estu} below) is of one-dimensional nature, and will provide a bound for $P_E (B - \tilde B) P_E$. Next, Lemma \ref{estup} will essentially be the only place in this note where the high dimension plays a role, and will give a bound for $\Vert  P_{E^\perp} B  P_{E^\perp} \Vert_{HS}$. Finally, in Lemma \ref{estoffdiag} which is of a two-dimensional nature, we give a bound for $\Vert  P_{E^\perp} B  P_{E} \Vert_{HS}$.

In all the proofs of this section, the time $t$ will be fixed, so the reader may consider the set $A_t$ as an arbitrary fixed measurable set. For convenience, we repeat a few definitions which will be used intensively in our proofs. First of all, recall that
$$
B_t = \int_{A_t} \left ( x \otimes x - \mathrm{Id} \right ) \gamma^n(x) dx
$$
and
$$
\tilde B_t = \int_{H(A_t)} \left ( x \otimes x - \mathrm{Id} \right ) d \gamma^n(x) = \frac{1}{\sqrt{2 \pi}} \left ( \int_{-\infty}^{\Psi(S_t)} (s^2 - 1) e^{-s^2 / 2} ds \right ) u_t \otimes u_t
$$
where
$$
S_t = \gamma^n(A_t) = \gamma^n (H(A_t))
$$
and
$$
u_t = \frac{\int_{H(A_t)} x d \gamma^n(x)}{\left |\int_{H(A_t)} x d \gamma^n(x) \right |.}
$$
Moreover,
\begin{equation} \label{recallq}
\eps_t = q(S_t)^2 - q(A_t)^2 = \left | \int_{H(A_t)} x d \gamma^n(x) \right |^2 - \left | \int_{A_t} x d \gamma^n(x) \right |^2.
\end{equation}
Finally, it will be useful to recall that
$$
H(A(t)) = \left \{x; ~~ \langle x, u_t \rangle \geq - \Psi(S_t) \right \}.
$$

\medskip
We begin with:

\begin{lemma} \label{estu}
For all $0 \leq t \leq 1$, one has
\begin{equation} \label{eqestu}
\left | \left \langle \left ( B_t - \tilde B_t \right ) u_t, u_t \right \rangle \right | \leq \frac{C}{S_t^2(1-S_t)^2} \epsilon_t \sqrt{\left |\log \epsilon_t \right | }
\end{equation}
for a universal constant $C>0$.
\end{lemma}
\begin{proof}
Let $f:\RR \to [0,1]$ be the unique continuous function satisfying, for all measurable subsets $W \subset \RR$,
$$
\int_W f(x) d \gamma^1(x) = \int_{A_t} \mathbf{1}_{ \{ \langle x, u_t \rangle \in W  \}  }(x) d \gamma^n(x)
$$
hence $f$ is the density with respect to the Gaussian measure of the marginal on $\mathrm{sp}\{u_t\}$ of the standard Gaussian measure restricted to the set $A_t$, and similarly define $h(x)$ by
$$
\int_W h(x) d \gamma^1(x) = \int_{H(A_t)} \mathbf{1}_{ \{ \langle x, u_t \rangle \in W  \}  }(x) d \gamma^n(x).
$$
By definition, we get
\begin{equation} \label{intf}
\int_{\RR} f(x) d \gamma^1(x) = \int_{\RR} h(x) d \gamma^1(x) = S_t
\end{equation}
and
$$
\langle B_t u_t, u_t  \rangle = \int_{\RR} (x^2 - 1) f(x) d \gamma^1(x), ~~ \langle \tilde B_t u_t, u_t  \rangle = \int_{\RR} (x^2 - 1) h(x) d \gamma^1(x).
$$
Next, define
$$
g(x) = h(x) - f(x).
$$
Equation (\ref{intf}) teaches us that $\int g(x) d \gamma^1(x) = 0$, which gives
\begin{equation} \label{con0}
\left \langle \left ( B_t - \tilde B_t \right ) u_t, u_t \right \rangle = \int_{\RR} (x^2 - 1) g(x) d \gamma^1(x) = \int_{\RR} x^2 g(x) d \gamma^1(x).
\end{equation}
We claim that in order to complete the proof, it is enough to show that
\begin{equation} \label{enough111}
\int_{\RR} x^2 g(x) d \gamma^1(x) \leq \frac{C}{S_t^2(1-S_t)^2} \epsilon_t \sqrt{\left |\log \epsilon_t \right | }
\end{equation}
for a universal constant $C>0$. Indeed, observe that for all $B \subset \RR^n$,  one has $\eps(B) = \eps(B^C)$. Consequently, the right hand side of formula (\ref{eqestu}) remains invariant if we replace $A_t$ by $A_t^C$. Therefore, if the left hand side of the above equation is negative, we may replace $g(x)$ with $-g(-x)$ which corresponds to replacing $A_t$ by $A_t^C$ and continue as usual. \\

Denote $\delta = \int x g(x) d \gamma^1(x)$. We have
$$
\delta = \int_{H(A_t)} \langle x, u_t \rangle  d \gamma(x)  - \int_{A_t} \langle x, u_t \rangle d \gamma(x) = q(S_t) - q(A_t) = \frac{\epsilon_t}{q(S_t) + q(A_t)}.
$$
which, together with the bound (\ref{estq1}) and the fact that $q(s) < 1$ for all $0 < s < 1$ implies that
\begin{equation} \label{xgx}
\epsilon_t / 2 \leq  \delta \leq \epsilon_t / q(S_t) \leq \frac{C}{S_t(1-S_t)} \epsilon_t
\end{equation}
for a universal constant $C>0$. Define
$$
p = 100 \sqrt{ |\log \delta|}.
$$
The fact that $\int_\RR (x^2 + 1) d \gamma^1(x) < \infty$ implies that the left hand side of (\ref{eqestu}) is always smaller than a universal constant, therefore we remark that if $\eps_t \geq S_t^2(1-S_t)^2$ then this formula holds trivially and there is nothing to prove. Consequently, we may assume that $\delta \leq S_t(1-S_t)$. A well-known estimate about the Gaussian distribution is
$$
|\Psi(S_t)| \leq 10 \sqrt{ | \log((S_t)(1-S_t))| } \leq 10 \sqrt{|\log \delta|}.
$$
And therefore,
\begin{equation} \label{factp}
p \geq 10 |\Psi(S_t)|.
\end{equation}
Define $L = - \Psi(S_t)$ so that $h(x) = \mathbf{1}_{x \geq L}$.
Clearly, $g(x) \geq 0$ for $x > L$ and $g(x) \leq 0$ for $x < L$.
We have
\begin{equation} \label{f1}
\int_{\RR} x^2 g(x) d \gamma^1(x) = \int_{\RR} (x-L)^2 g(x) d \gamma^1(x) + 2 L \delta \leq
\end{equation}
$$
\int_L^\infty (x-L)^2 g(x) d \gamma^1(x) + 2 p \delta
$$
and also
\begin{equation} \label{leqdelta}
\int_{L}^\infty (x-L) g(x) d \gamma^1(x) \leq \delta.
\end{equation}
First, we estimate
$$
\int_p^\infty (x-L)^2 g(x) d \gamma^1(x) \leq \int_p^\infty (x-L)^2 d \gamma^1(x) \leq
$$
(according to equation (\ref{factp}))
$$
\int_p^\infty (x+p)^2 d \gamma^1(x) \leq 4 \int_p^\infty x^2 d \gamma^1(x) = \frac{4}{\sqrt{2 \pi}} \int_p^\infty x^2 e^{-x^2/2} dx =
$$
(integration by parts)
$$
4 p e^{-p^2/2} + 4 ( 1 -\Phi(p)) \leq 10 p e^{-p^2/2} \leq  \delta^2.
$$
(where we use the legitimate assumption that $\delta$ is smaller than some universal constant, justified above). On the other hand, using (\ref{factp}) and (\ref{leqdelta}), we have
$$
\int_L^p (x-L)^2 g(x) d \gamma^1(x) \leq (p-L) \int_L^p (x-L) g(x) d \gamma^1(x) \leq 2 p \delta.
$$
The last two equations with (\ref{f1}) give
$$
\int_{\RR} x^2 g(x) d \gamma^1(x) \leq 5 p \delta = 500 \delta \sqrt{ |\log \delta|}.
$$
Equation (\ref{xgx}) now tells us that
$$
\int_{\RR} x^2 g(x) d \gamma^1(x) \leq \frac{2C}{S_t(1-S_t)} \epsilon_t \sqrt{|\log \epsilon_t|}.
$$
Thus, equation (\ref{enough111}) holds and the proof is complete.
\end{proof}

Recall that we denote by $P_{E^\perp}$ the orthogonal projection onto $E^\perp=\mathrm{sp}\{u_t\}^\perp$. Next, we would like to prove
\begin{lemma} \label{estup}
For all $0 < t < 1$,
$$
\Vert P_{E^\perp} B_t P_{E^\perp} \Vert_{HS}^2 \leq \frac{C}{S_t^2(1-S_t)^2} \epsilon_t
$$
where $C>0$ is a universal constant.
\end{lemma}

Before we prove this lemma, we first need

\begin{lemma} \label{entropy}
Let $0 < h \leq 1$ and let $f: \Rnn \to [0,1]$ be such that
$$
\int_{\Rnn} f(x) d \gamma^{n-1}(x) = h.
$$
Then
$$
\int_{\Rnn} \frac{f(x)}{h} \log \left (\frac{f(x)}{h}  \right ) d \gamma^{n-1}(x) \leq h^{-2} \int_{\Rnn} \left (q(h) - q(f(x)) \right ) d \gamma^{n-1}(x).
$$
\end{lemma}

For this lemma we will need the following fact, whose simple yet technical proof is postponed to the appendix.
\begin{fact} \label{secondder}
For all $0 \leq h,s \leq 1$,
$$
- \frac{4}{h^2(1-h)^2} (s-h)^2 \leq q(s) - q(h) - q'(h) (s-h) \leq - (s-h)^2.
$$
where $C>0$ is a universal constant.
\end{fact}

\begin{proof} [\textbf{Proof of Lemma \ref{entropy}}]
An application of Fact \ref{secondder} gives
\begin{equation} \label{onehand}
\int_{\Rnn} \left (q(h) - q(f(x)) \right ) d \gamma^{n-1}(x) = \int_{\Rnn}  \left (q(h) - q(f(x)) + q'(h) (f(x) - h) \right ) d \gamma(x) \geq
\end{equation}
$$
\int_{\Rnn} (f(x)-h)^2 d \gamma^{n-1}(x) = h^2 \int_{\Rnn} \left (\dfrac{f(x)}{h} - 1 \right )^2 d \gamma^{n-1}(x).
$$
On the other hand, it is easy to check that one has,
$$
s \log s - (s-1) \leq (s-1)^2
$$
for all $s \geq 0$. Consequently,
\begin{equation} \label{otherhand}
\int_{\Rnn} \frac{f(x)}{h} \log \left ( \frac{f(x)}{h} \right ) d \gamma^{n-1}(x) =
\int_{\Rnn} \left (\frac{f(x)}{h} \log \left ( \frac{f(x)}{h} \right ) - \left (\frac{f(x)}{h} - 1 \right ) \right ) d \gamma^{n-1}(x) \leq
\end{equation}
$$
\int_{\Rnn} \left (\dfrac{f(x)}{h} - 1 \right )^2 d \gamma^{n-1}(x).
$$
A combination of (\ref{onehand}) and (\ref{otherhand}) finishes the proof.
\end{proof}

\begin{proof} [\textbf{Proof of Lemma \ref{estup}}]
We begin with the observation that
$$
B_t = \int_{A_t} \left ( x \otimes x - \mathrm{Id} \right ) d \gamma^1(x) = - \int_{A_t^C} \left ( x \otimes x - \mathrm{Id} \right ) d \gamma^1(x).
$$
Moreover $q(A_t) = q(A_t^C)$. Thanks to this, the statement of the lemma remains invariant when replacing that set $A_t$ with the set $A_t^C$. Consequently, it is legitimate to make the assumption
\begin{equation} \label{symasum}
S_t \geq \frac{1}{2}.
\end{equation}
Let $\mu$ be the measure $\gamma^n \big \vert_{A_t}$, hence the Gaussian measure restricted to the set $A_t$. Define by $\tilde \mu$ and $\tilde \gamma$ the push-forward under $P_{E^\perp}$ of the measures $\mu$ and $\gamma^n$ respectively. Define the function $f:E^\perp \to \RR$ by,
$$
f(y) = \frac{1}{\sqrt{2 \pi}} \int_{(y + E) \cap A_t} \exp \left (- \langle x, u_t \rangle^2 / 2 \right ) dx
$$
(here $dx$ stands for the $1$-dimensional Hausdorff measure on $y+E$). One can verify that for $y \in E^\perp$, this function satisfies
$$
\frac{d \tilde \mu}{d \tilde \gamma} (y) = f(y).
$$
Thus $\int_{E^\perp} f(y) d \tilde \gamma(y) = \gamma(A_t) = S_t$. An application of Lemma \ref{entropy} and of equation (\ref{symasum}) now gives
\begin{equation} \label{entropyappl}
\int_{E^\perp} \frac{f(y)}{S_t} \log \left ( \frac{f(y)}{S_t} \right ) d \tilde \gamma(y) \leq 4 \left ( q(S_t) - \int_{E^\perp} q(f(y)) d \tilde \gamma(y) \right ).
\end{equation}
Now, it follows from Claim \ref{claim1} that for all $y \in E^\perp$,
$$
 \frac{1}{\sqrt{2 \pi}} \left | \int_{(y + E) \cap A_t} \langle x, u_t \rangle  \exp(- \langle x, u_t \rangle^2 / 2 ) dx \right | \leq
 $$
 $$
 q \left (  \frac{1}{\sqrt{2 \pi}}  \int_{(y + E) \cap A_t} \exp(- \langle x, u_t \rangle^2 / 2 )  dx  \right ) = q(f(y)).
$$
Integrating this inequality over $E^\perp$ with respect to $\tilde \gamma$ and using equation (\ref{recallq}) gives
$$
\int_{E^\perp} q(f(y)) d \tilde \gamma(y) \geq q(A_t).
$$
Combining this with equation (\ref{entropyappl}) gives
\begin{equation}
\int_{E^\perp} \frac{f(y)}{S_t} \log \left ( \frac{f(y)}{S_t} \right ) d \tilde \gamma(y) \leq 4 ( q(S_t) - q(A_t)) \leq \frac{4}{q(S_t)} \epsilon_t \leq \frac{C}{S_t(1-S_t)} \epsilon_t
\end{equation}
for a universal constant $C>0$ (in the last inequality we used formula (\ref{estq1})).
The above equation allows us to use Talagrand's transportation-entropy inequality  (\cite{T}) which teaches us that there exists a function $T: \EE^{\perp} \to \EE^\perp$ which pushes forward the measure $\tilde \gamma$ to the measure $S_t^{-1} \tilde \mu$ and such that,
$$
\int_{E^\perp} |T(y) - y|^2 d \tilde \gamma(y) \leq \frac{2 C}{S_t(1-S_t)} \epsilon_t.
$$
Denote $D = P_{E^\perp} B_t P_{E^\perp}$. Let $X$ be a random vector whose law is $\tilde \gamma$, then by definition
$$
D = S_t (Cov(T(X)) - Cov(X))
$$
where $Cov(Y) := \EE[ (Y - \EE[Y]) \otimes (Y - \EE[Y])]$ is the covariance matrix of a vector $Y$ (here, we use the fact that $\int_{A_t} P_{E^\perp} x d \gamma(x) = 0$). Let $e_1,...,e_{n-1}$ be an orthogonal basis of $E^\perp$ which diagonalizes $D$. We have, for all $1 \leq i \leq n-1$,
$$
Var \left (\langle X + T(X) , e_i \rangle \right  ) \leq 8 Var( \langle X, e_i \rangle ) + 2 Var \left (\langle X - T(X) , e_i \rangle \right  ) \leq
$$
$$
8 + 2 \frac{C}{S_t(1-S_t)} \epsilon_t \leq \frac{C'}{S_t(1-S_t)}
$$
for a universal constant $C'>0$. We calculate,
$$
\Vert D \Vert_{HS}^2 = S_t^2 \sum_{i=1}^n \left (Var( \langle T(X), e_i \rangle ) - Var( \langle X, e_i \rangle ) \right )^2 =
$$
$$
S_t^2 \sum_{i=1}^n Cov( \langle T(X) - X, e_i \rangle , \langle T(X) + X , e_i \rangle  )^2 \leq
$$
$$
S_t^2 \sum_{i=1}^n Var( \langle T(X) - X, e_i \rangle) Var( \langle T(X) + X , e_i \rangle  ) \leq
$$
$$
\frac{C}{S_t(1-S_t)} \EE | T(X) - X |^2 \leq \frac{C}{S_t^2(1-S_t)^2} \epsilon_t,
$$
and we are done.
\end{proof}

Our last lemma concerns with the part of the matrix $B$ which is "off-diagonal" with respect to $E,E^\perp$.

\begin{lemma} \label{estoffdiag}
For any $0 \leq t < 1$ and for any unit vector $v$ with $v \perp u_t$, one has
$$
\left | \langle v, B_t u_t \rangle \right | \leq \frac{C}{S_t^3(1 - S_t)^3} \epsilon_t \sqrt{|\log \epsilon_t| }
$$
where $C>0$ is a universal constant.
\end{lemma}
\begin{proof}
Let $\mu$ be the measure whose density is $\gamma^n \big \vert_{A_t}$. Denote $F = \mathrm{sp} \{v, u_t \}$, and let $\tilde \mu$ and $\tilde \gamma$ be the push-forward of $\mu$ and $\gamma$ under the orthogonal projection onto $F$ respectively. Define,
$$
f(x,y) = \frac{d \tilde \mu}{d \tilde \gamma} (x u_t + y v), ~~~ \forall (x,y) \in \RR^2.
$$
By definition, we have
\begin{equation} \label{havetobound}
\langle v, B_t u_t \rangle = \int_{A_t} \langle x, v \rangle \langle x, u_t \rangle d \gamma(x) = \int_{\RR^2} x y f(x,y) d \gamma^2(x,y).
\end{equation}
so our objective is to bound the right hand side of the above equation. For every $y \in \RR$, we write
$$
g(y) = \int_{\RR} f(x,y) d \gamma^1(x),
$$
the density of the marginal of $f(x,y)$ onto the $y$ coordinate with respect to the Gaussian measure, and
$$
Q(y) = \int_{\RR} x f(x,y) d \gamma^1(x).
$$
Equation (\ref{havetobound}) becomes
\begin{equation} \label{havetobound2}
\langle v, B_t u_t \rangle = \int_{\RR} y Q(y) d \gamma^1(y).
\end{equation}
By equation (\ref{recallq}), we know that
$$
q(A_t) = \left | \int_{A_t} \langle x, u_t \rangle d \gamma(x) \right | = \left | \int_{\RR^2} x f(x,y) d \gamma^2(x,y) \right | = \left | \int_{\RR} Q(y) d \gamma^1(y) \right |.
$$
So, by the definition of $\epsilon_t$,
\begin{equation} \label{Qsmall}
q^2(S_t) - \left | \int_{\RR} Q(y) d \gamma^1(y) \right |^2 = \epsilon_t.
\end{equation}
Since $|f(x,y)| \leq 1$ for all $(x,y) \in \RR^2$ and by Lemma \ref{claim1alt}, we have
\begin{equation} \label{Qsmall2}
|Q(y)| \leq q(g(y)), ~ \forall y \in \RR.
\end{equation}
The two above equations give
$$
q^2(S_t) - \left ( \int_{\RR} q(g(y)) d \gamma^1(y) \right )^2 \leq \epsilon_t,
$$
and therefore
$$
q(S_t) - \int_{\RR} q(g(y)) d \gamma^1(y) \leq \epsilon_t / q(S_t).
$$
Next, observe that $\int_{\RR} g(y) d \gamma = S_t$. Using Fact \ref{secondder}, this yields
\begin{equation}
q(S_t) - \int_{\RR} q(g(y)) d \gamma^1(y) = - \int_{\RR} (q(g(y)) - q(S_t) - q'(S_t) (g(y) - S_t) ) d \gamma^1(y) \geq
\end{equation}
$$
\int_{\RR} (g(y) - S_t)^2 d \gamma^1(y),
$$
so
\begin{equation} \label{quadleqe}
\int_{\RR} (g(y) - S_t)^2 d \gamma^1(y) \leq \epsilon_t / q(S_t).
\end{equation}
Next, by definition of the vector $u_t$, we know that the center of mass of $\tilde \mu$ is orthogonal to $v$, which implies that $\int_{\RR} y g(y) d \gamma^1(y) = 0$. This gives,
\begin{equation}
\int_{\RR} y q(g(y)) d \gamma^1(y) = \int_{\RR} y (q(g(y)) - q(S_t) - q'(S_t)(g(y) - S_t)) d \gamma^1(y)
\end{equation}
and, by Fact \ref{secondder}
\begin{equation} \label{qsep}
\left | \int_{\RR} y q(g(y)) d \gamma^1(y) \right | \leq \frac{4}{S_t^2(1-S_t)^2} \int_{\RR} |y| (g(y) - S_t)^2 d \gamma^1(y).
\end{equation}
We claim that the last equation combined with (\ref{quadleqe}) gives,
\begin{equation} \label{qse}
\left | \int_{\RR} y q(g(y)) d \gamma^1(y) \right | \leq \frac{4}{S_t^2(1-S_t)^2} q(\epsilon_t) / q(S_t).
\end{equation}
Indeed, observe that for all $y$, the quantity $(g(y) - S_t)^2$ is smaller than $1$. We invoke Lemma \ref{claim1alt} with $m(y) = (g(y) - S_t)^2$ and use the bound (\ref{quadleqe}) to get
$$
\int_{\RR} |y| (g(y) - S_t)^2 d \gamma^1(y) =
\int_0^\infty y (g(y) - S_t)^2 d \gamma^1(y) -
\int_{-\infty}^0 y (g(y) - S_t)^2 d \gamma^1(y) \leq
$$
$$
 2 \int_{-\infty}^{\Psi(\epsilon_t / q(S_t))} |y| d \gamma^1(y) = 2 q(\epsilon_t / q(S_t)).
$$
In the last equality, we have used the legitimate assumption that $\epsilon_t < \frac{1}{2}$. Equation (\ref{qse}) now follows from the sub-linearity of $q(\cdot)$ suggested by equation (\ref{derq1}).\\
Next, another application of Claim \ref{claim1} on the set
$$
\{(x,y); ~~ x \leq \Psi(g(y)) \}
$$
teaches us that
$$
q(S_t) - \int_{\RR} q(g(y)) d \gamma^1(y) \geq 0.
$$
Combining this fact with (\ref{Qsmall}) suggests that
\begin{equation} \label{beforeeps}
0 \leq \left ( \int_{\RR} q(g(y)) d \gamma^1(y) \right )^2 - \left ( \int_{\RR} Q(y) d \gamma^1(y) \right )^2 \leq \epsilon_t.
\end{equation}
Next, we note that the assumption
\begin{equation} \label{epssmall1}
\epsilon_t \leq q(S_t) / 2
\end{equation}
is a legitimate one. Indeed, one has $\int_{\RR^2} |x y| d \gamma^2(x,y) < \infty $ which, thanks to equation (\ref{havetobound}) teaches us that
$$
\langle v, B_t u_t \rangle \leq C_1
$$
for some universal constant $C_1>0$. The estimate (\ref{estq1}) ensures us that if $\epsilon_t \geq q(S_t) / 2$ then the quantity $\frac{q(\epsilon_t)}{S_t(1-S_t)}$ is larger than a universal constant, which would imply the result of the lemma, so the assumption can be made. Using assumption (\ref{epssmall1}) with equation (\ref{beforeeps}) yields
$$
\int_{\RR} \left ( q(g(y)) - Q(y) \right ) \ d \gamma^1(y) \leq \frac{\epsilon_t}{\int_{\RR}  q(g(y)) d \gamma^1(y)} \leq 2 \epsilon_t / q(S_t).
$$
According to $(\ref{Qsmall2})$ and since $q(s) \leq q(1/2)$ for $0 < s < 1$, we have $0 \leq q(g(y)) - Q(y) \leq 1$. Thus, in a similar way that (\ref{qsep}) implied (\ref{qse}), the above equation implies
$$
\left | \int_{\RR} y (q(g(y)) - Q(y)) d \gamma^1(y) \right | \leq 4 q(\epsilon_t) / q(S_t).
$$
Finally, combine the above equation with (\ref{havetobound2}) and (\ref{qse}) to get
$$
\left | \langle v, B_t u_t \rangle \right | = \left | \int_{\RR} y Q(y) d \gamma^1(y) \right | \leq
$$
$$
\left | \int_{\RR} y (Q(y) - q(g(y)) d \gamma^1(y) \right | + \left | \int_{\RR} y q(g(y)) d \gamma^1(y) \right | \leq
$$
$$
\left (\frac{4}{S_t^2(1-S_t)^2} + 4 \right ) q(\epsilon_t) / q(S_t).
$$
Using the estimate (\ref{estq1}) completes the proof.
\end{proof}
\bigskip

We are now ready to prove the main proposition of the section.
\begin{proof} [\textbf{Proof of Proposition \ref{mainsec4}}]
The proof is just a combination of the lemmas in the section together with equation (\ref{finald}), which reads
\begin{equation}
d \epsilon_t = d \left ( q(S_t)^2 - |U_t|^2 \right ) =
\end{equation}
$$
2 (1-t)^{-1/2} |U_t| \left  \langle \left (\tilde B_t - B_t \right) u_t , d W_t \right \rangle - (1-t)^{-1} \Vert B_t \Vert^2_{HS} dt + (1-t)^{-1} \left \Vert \tilde B_t \right \Vert^2_{HS} \frac{|U_t|^2}{q(S_t)^2} dt.
$$
Denote,
$$
\alpha_t = 2 (1-t)^{-1/2} |U_t| \left (\tilde B_t - B_t \right) u_t
$$
and
$$
\beta_t = (1-t)^{-1} \left (\left \Vert \tilde B_t \right \Vert^2_{HS} \frac{|U_t|^2}{q(S_t)^2} - \Vert B_t \Vert^2_{HS} \right )
$$
so that
$$
d \epsilon_t =  \langle \alpha_t, d W_t \rangle + \beta_t dt.
$$
Since $|U_t| \leq q(1/2)$, in order to prove part (i) of the proposition it suffices to show that
\begin{equation} \label{ntsb1}
\left | \left (B_t - \tilde B_t \right) u_t \right | < \frac{C}{S_t^3(1 - S_t)^3} \epsilon_t \sqrt{|\log \epsilon_t|}
\end{equation}
for a universal constant $C>0$. By the triangle inequality,
$$
\left | \left (B_t - \tilde B_t \right) u_t \right | \leq
$$
$$
\left | \left \langle \left ( B_t - \tilde B_t \right ) u_t, u_t \right \rangle \right | + \max_{v \perp u_t, \atop |v| = 1} |\langle v, B_t u_t \rangle|.
$$
A combination of lemmas \ref{estu} and \ref{estoffdiag} establishes (\ref{ntsb1}). \\
In order to prove part (ii) of the proposition, we write
$$
\Vert B \Vert_{HS}^2 = \Vert P_{E^\perp} B_t P_{E^\perp} \Vert_{HS}^2 + \max_{v \perp u_t, \atop |v| = 1} |\langle v, B_t u_t \rangle|^2 + \left | \left \langle  B_t u_t, u_t \right \rangle \right |^2.
$$
Also, by definition of the rank-one matrix $\tilde B_t$, we have
$$
\Vert \tilde B_t \Vert_{HS}^2 = \left | \left \langle \tilde B_t u_t, u_t \right \rangle \right |^2.
$$
These two facts combined and the triangle inequality give
\begin{equation} \label{HSeq}
(1-t) |\beta_t| =  \left | \left (\Vert \tilde B \Vert_{HS}^2 - \Vert B \Vert_{HS}^2 \right )+
\Vert \tilde B \Vert_{HS}^2 \left (\frac{|U_t|^2}{q(S_t)^2} - 1 \right)  \right | \leq
\end{equation}
$$
\Vert P_{E^\perp} B_t P_{E^\perp} \Vert_{HS}^2 + \max_{v \perp u_t, \atop |v| = 1} |\langle v, B_t u_t \rangle|^2 +
$$
$$
\left | \left \langle \left ( B_t - \tilde B_t \right ) u_t, u_t \right \rangle \right | \left | \left \langle \left ( B_t + \tilde B_t \right ) u_t, u_t \right \rangle \right | +
$$
$$
\frac{1}{q(S_t)^2} \left | \left \langle \tilde B_t u_t, u_t \right \rangle \right |^2 \epsilon_t
$$
We turn to estimate each term separately. First we remark that by the triangle inequality
\begin{equation} \label{bsmall}
B_t^2  \leq \left  ( \int_{\RR^n} x \otimes x d \gamma(x) + \mathrm{Id} \right )^2 \leq 4 \mathrm{Id}.
\end{equation}
Therefore,
$$
\left | \left \langle B_t u_t, u_t \right \rangle \right | \leq 2
$$
and, analogously,
\begin{equation} \label{tildebsmall}
\left | \left \langle \tilde B_t u_t, u_t \right \rangle \right | \leq 2.
\end{equation}

These two equations together with Lemma \ref{estu} give
\begin{equation} \label{secondterm}
\left | \left \langle \left ( B_t - \tilde B_t \right ) u_t, u_t \right \rangle \right | \left | \left \langle \left ( B_t + \tilde B_t \right ) u_t, u_t \right \rangle \right | \leq \frac{C}{S_t^2(1-S_t)^2} \epsilon_t \sqrt{|\log \epsilon_t|}.
\end{equation}
Equation (\ref{bsmall}) also teaches us that $|\langle v, B_t u_t \rangle|^2 \leq 4 |\langle v, B_t u_t \rangle|$. We now use Lemma \ref{estoffdiag} and Lemma \ref{estup} which together give
\begin{equation} \label{firsttemrs}
\Vert P_{E^\perp} B_t P_{E^\perp} \Vert_{HS}^2 + \max_{v \perp u_t, \atop |v| = 1} |\langle v, B_t u_t \rangle|^2 \leq \frac{C}{S_t^3(1-S_t)^3} \epsilon_t \sqrt{|\log \epsilon_t|}.
\end{equation}
Another application of (\ref{tildebsmall}) together with equation (\ref{estq1}) gives
\begin{equation} \label{thirdterm}
\frac{1}{q(S_t)^2} \left | \left \langle \tilde B_t u_t, u_t \right \rangle \right |^2 \epsilon_t \leq \frac{C}{S_t^2(1-S_t)^2} \epsilon_t.
\end{equation}
Finally, plugging the estimates (\ref{secondterm}), (\ref{firsttemrs}) and (\ref{thirdterm}) into (\ref{HSeq}) gives
$$
|\beta_t| \leq (1-t)^{-1} \frac{C}{S_t^3(1-S_t)^3} \epsilon_t \sqrt{|\log \epsilon_t|}
$$
and the proof is complete.
\end{proof}

\subsection{The upper bound} \label{subupper}

The goal of this subsection is to prove the upper bound for the deficit in Theorem \ref{thmrobust}. Namely, we aim to show that for all $0 < s < 1$ there exists a constant $c_s>0$, such that the following holds: for every $0 < \rho < 1$ and every measurable $A \subset \RR^n$, one has

\begin{equation} \label{uppernts}
\St_\rho (H(A)) - \St_\rho(A) \leq \frac{c_{\gamma(A)} }{\sqrt{1 - \rho}} \eps(A).
\end{equation}

We begin defining the operator $P_\rho$ acting on integrable functions $f:\RR^n \to \RR$ by the formula
\begin{equation} \label{defOU}
P_\rho(f) (x) = \int_{\RR^n} f \left ( \rho x + \sqrt{1-\rho^2} y  \right) d \gamma(y).
\end{equation}
Under a slightly different parametrization, this is just the Ornstein-Uhlenbeck
operator for the Gaussian measure. For a measurable set $A \subset \RR^n$ we abbreviate $P_\rho(A) = P_\rho(\mathbf{1}_A)$.
The significance of this definition is the following: for any two measurable sets $A,B \subset \RR^n$, we have
\begin{equation} \label{Pt}
\int_A P_\rho(B) (x) d \gamma (x) =
\end{equation}
$$
\int_{\RR^n} \int_{\RR^n} \mathbf{1}_{x \in A} \mathbf{1}_{ \rho x + \sqrt{1-\rho^2} y \in B} d \gamma(y) d \gamma(x) =
$$
$$
\PP( X \in A ~ \mbox{ and } ~\rho X + \sqrt{1-\rho^2} Y \in B )
$$
where $X$ and $Y$ are independent standard Gaussian vectors.

Now, if $Y'$ is another standard Gaussian random vector independent from $X,Y$, then it is easy to check that
$$
( \sqrt{\rho} X+ \sqrt{1-\rho}Y, \sqrt{\rho} X+ \sqrt{1-\rho}Y' ) \sim (X, \rho X+ \sqrt{1-\rho^2}Y ).
$$
where the sign $\sim$ means that both expressions are distributed according to the same law. Consequently, we get by definition that
\begin{equation} \label{Stpf}
\St_\rho(A) = \int_A P_\rho(A) d \gamma
\end{equation}
for every $A \subset \RR^n$ measurable. Moreover, since the right hand side of (\ref{Pt}) is invariant under interchanging $A$ and $B$, we learn that $P_\rho$ is a self-adjoint linear operator. It is straightforward to check that for all $f$,
$$
P_\rho(f) =P_{\sqrt \rho} (P_{\sqrt \rho}(f))
$$
and it follows that $P_\rho$ is a positive semi-definite operator. Consider the non-negative, symmetric quadratic form
$$
K_\rho(f,g) := \int_{\RR^n} f(x) \left ( P_\rho(g) (x) \right ) d \gamma(x).
$$
By formula (\ref{Stpf}) we have
$$
\St_\rho(H(A)) - \St_\rho(A) = K_\rho(H(A),H(A)) - K_\rho(A,A) =
$$
\begin{equation} \label{ekrho}
(K_\rho(H(A),H(A)) - K_\rho(H(A),A)) + (K_\rho(H(A),A) - K_\rho(A,A)).
\end{equation}
We claim that in order to give an upper bound for the deficit, it is enough to estimate the first term in the above equation. Namely, we claim that
\begin{equation} \label{indeed}
\St_\rho(H(A)) - \St_\rho(A) \leq 2 (K_\rho(H(A),H(A)) - K_\rho(H(A),A)).
\end{equation}
Indeed, according to Theorem \ref{thm1}, we have
$$
K_\rho(H(A),H(A)) \geq K_\rho(A,A).
$$
By the Cauchy-Schwartz and the arithmetic-geometric inequalities
$$
K_\rho (H(A), A) \leq \sqrt{ K_\rho (H(A), H(A)) K_\rho(A,A)} \leq \frac{K_\rho (H(A), H(A)) + K_\rho(A,A)}{2},
$$
or in other words
$$
(K_\rho(H(A),H(A)) - K_\rho(H(A),A)) \geq  (K_\rho(H(A),A) - K_\rho(A,A)).
$$
Plugging this into (\ref{ekrho}) gives (\ref{indeed}). Let us give an upper bound for the right hand side of $(\ref{indeed})$. We have by definition
$$
I := K_\rho(H(A),H(A)) - K_\rho(H(A),A) =
$$
$$
\int_{\RR^n}  P_\rho(H(A)) (x) (\mathbf{1}_{H(A)} - \mathbf{1}_A ) (x) d \gamma(x).
$$
Let $v$ be a unit vector, and $\alpha \in \RR$ such that
$$
H(A) = \{\langle x, v \rangle \geq \alpha \}.
$$
Moreover, let $\mu$ be the push-forward of the restriction of the standard Gaussian measure to the set $A$ under the map $x \to \langle x, v \rangle$, let $f(x)$ be the density of $\mu$ with respect to the standard Gaussian measure and define $h(x) = \mathbf{1}_{x \geq \alpha}$. Since $H(A)$ is invariant under translations orthogonal to $v$, it is clear that
$$
P_\rho(H(A)) (x) = P_\rho(h) (\langle x, v \rangle).
$$
With this notation, the above integral becomes
$$
I = \int_{\RR}  P_\rho(h) (x) (h(x) - f(x)) d \gamma^1(x).
$$
We can calculate,
$$
P_\rho(h) (x) = \int_{\RR} \mathbf{1}_{\rho x + \sqrt{1 - \rho^2} y \geq \alpha} d \gamma(y) =
$$
$$
\gamma \left ( \left [\frac{\alpha - \rho x}{\sqrt{1 - \rho^2}}, \infty \right ) \right  ) = \Phi \left ( \frac{\rho x - \alpha}{\sqrt{1-\rho^2}} \right )
$$
and the above integral becomes
$$
I = \int_{\RR} \Phi \left ( \frac{\rho x - \alpha}{\sqrt{1-\rho^2}} \right ) (h(x) - f(x)) d \gamma^1(x).
$$
Since $\gamma(A) = \gamma(H(A))$, we know that
$$
\int_{\RR} (h(x) - f(x)) d \gamma^1(x) = 0
$$
and therefore
\begin{equation} \label{inspiration}
I = \int_{\RR} \left ( \Phi \left ( \frac{\rho x - \alpha}{\sqrt{1-\rho^2}} \right ) - \Phi \left ( \frac{\rho \alpha - \alpha}{\sqrt{1-\rho^2}} \right ) \right ) (h(x) - f(x)) d \gamma^1(x).
\end{equation}
Since $\Phi(x)' \leq 1$ for all $x \in \RR$, we have
\begin{equation}
\left | \Phi \left ( \frac{\rho x - \alpha}{\sqrt{1-\rho^2}} \right ) - \Phi \left ( \frac{\rho \alpha - \alpha}{\sqrt{1-\rho^2}} \right ) \right | \leq \frac{\rho}{\sqrt{1 - \rho^2}} |x - \alpha|.
\end{equation}
Next, observe that by definition $h(x) - f(x) \geq 0$ for $x \geq \alpha$ and $h(x) - f(x) \leq 0$ for $x \leq \alpha$. Since $\Phi(\cdot)$ is an increasing function, it implies that the expression inside the above integral is non-negative. Therefore, we can estimate
$$
I \leq \frac{\rho}{\sqrt{1 - \rho^2}} \int_{\RR} (x - \alpha) (h(x) - f(x)) d \gamma^1(x) = \frac{\rho}{\sqrt{1 - \rho^2}} (q(H(A)) - q(A)) \leq
$$
$$
\frac{\rho}{\sqrt{1 - \rho^2}} \frac{q(H(A))^2 - q(A)^2}{q(H(A)} \leq \frac{1}{\sqrt{1 - \rho}} \frac{C \eps(A)}{\gamma(A)(1-\gamma(A))}
$$
where the last inequality follows from the bound (\ref{estq1}). By plugging this into (\ref{indeed}), we get (\ref{uppernts}) and the upper bound is established.

\section{Appendix}
In the appendix we fill in a few technical lemmas whose proofs were omitted from the note.

\begin{proof}[\textbf{Proof of Lemma \ref{LemdiffF}}]
Define
$$
g_{x,t} (y) := \gamma_{y,\sqrt{1-t}} (x) = \frac{1}{(2 \pi (1-t))^{n/2} } \exp \left ( - \frac{|x-y|^2}{2 (1-t)} \right ).
$$
A simple calculation gives
$$
\nabla g_{x,t}(y) = \frac{(x-y)}{(1-t)} g_{x,t}(y),
$$
and therefore
$$
\Delta g_{x,t} (y) = \left (\frac{ |x-y|^2}{(1-t)^2 } + \frac{n}{(1-t)} \right ) g_{x,t} (y).
$$
Moreover,
$$
\frac{\partial}{\partial t} g_{x,t} (y) = - \left (\frac{ |x-y|^2}{2 (1-t)^2 } + \frac{n}{2 (1-t)} \right ) g_{x,t} (y).
$$
We can therefore calculate, using It\^{o}'s formula,
$$
d F_t(x) = d g_{x,t} (W_t) = \frac{\partial}{\partial t} g_{x,t} (W_t) + \nabla g_{x,t} W_t \cdot d W_t + \frac{1}{2} \Delta g_{x,t} (W_t)
$$
$$
=  \nabla g_{x,t} (W_t) \cdot d W_t = (1-t)^{-1} \langle x - W_t, d W_t \rangle F_t(x)
$$
which proves that $F_t(x)$ is a local martingale and establishes equation \eqref{stochastic}. 

Next, let $\phi: \RR^n \to \RR$ satisfy $|\phi(x)| < C_1 + C_2 |x|^p$ for some constants $C_1,C_2,p>0$. Remarking that the integral $\int_{\RR^n} \phi(x) \exp(-\alpha |x-x_0| ^2)$ is absolutely convergent for all $\alpha > 0$ and $x_0 \in \RR^n$, we deduce that for all $0<t<1$ and all $y \in \RR^n$, we have
$$
\nabla \int_{\RR^n} \phi(x) g_{x,t} (y) dx = \int_{\RR^n} \phi(x) \nabla g_{x,t}(y) dx
$$
and
$$  
\left (\frac{\partial}{\partial t} - \frac{1}{2} \Delta \right ) \int_{\RR^n} \phi(x) g_{x,t}(y) dx = 0.
$$
Formula \eqref{stochasticfub} follows. Finally, the fact that the process $t \to \int_{\RR^n} \phi(x) F_t(x) dx$ is a martingale follows immediately from the fact that
$$
\int_{\RR^n} \phi(x) F_t(x) dx = \EE[\phi(W_1) | \FF_t].
$$
\end{proof}
\bigskip
\begin{proof} [\textbf{Proof of Lemma \ref{lemestq}}]
We begin with formula (\ref{estq2}). By equation (\ref{q2}) we have
$$
\frac{q(s)}{s} = \frac{e^{-\Psi(s)^2 / 2}}{\int_{- \Psi(s)}^\infty e^{-x^2/2}dx }.
$$
Denote $y = - \Psi(s)$. Since, by (\ref{derq1}), $q'(s)$ is a decreasing function, we may assume that $s < \frac{1}{2}$ and
thus $y > 0$. The inequality $\left ( y + \frac{1}{y+1} \right )^2 \leq y^2 + 3$ suggests that
$$
\int_{y}^\infty e^{-x^2 / 2} \geq \int_{y}^{y + 1/(y+1)} e^{-(y + 1/(y+1))^2 / 2} dx \geq e^{-3} \frac{1}{y+1} e^{-y^2/2}
$$
so
$$
\frac{q(s)}{s} = \frac{e^{-y^2 / 2}}{\int_{y}^\infty e^{-x^2/2}dx } \leq e^{3} (y + 1) = - e^{3} (\Psi(s) + 1)
$$
for all $s < 1/2$. But a well known fact about the Gaussian distribution is that for $s < 1/2$
$$
- \Psi(s) \leq C \sqrt{|\log s|}
$$
for some a universal constant $C>0$. Formula (\ref{estq2}) follows. \\

The upper bound of formula (\ref{estq1}) now follows immediately from the symmetry of the function $q(s)$ around $s=1/2$, and we are left with proving the lower bound. Consider the function
$$
h(s) = 4 s(1-s) q(1/2) = \frac{4}{\sqrt{2 \pi}} s(1-s).
$$
We know that $h(s) = q(s)$ for $s \in \{0,1/2,1\}$. Moreover, $h(s)$ is tangent to $q(s)$ at $s = 1/2$,
and lastly, according to formula (\ref{derq2}), we see that $q'(s)$ is a convex function in $s \in [0,1/2]$. Consequently,
the convex function $g(s) = q'(s) - h'(s)$ intersects the x-axis exactly once in the interval $(0,1/2)$, say at the point $s_0$ (since it is equal to zero at $s=1/2$ and since its integral on that interval is equal zero). Now, we have
$$
q''(1/2) = - \sqrt{2 \pi} > - \frac{8}{\sqrt{2 \pi}} = h''(1/2),
$$
which implies that $g'(1/2) > 0$. We conclude that $g(s)(s-s_0) < 0$ for $0 < s < 1/2$. By the fact that $q(0) = h(0)$ and $q(1/2) = h(1/2)$ we know that
$$
\int_0^{1/2} g(s) ds = 0
$$
and therefore
$$
q(s) - h(s) = - \int_{s}^{1/2} g(x) dx \geq 0, ~~ \forall 0 < s < 1/2
$$
so $q(s) \geq h(s)$ in $0 < s < 1/2$. Since both functions are symmetric around $s=1/2$, we have established that
$$
q(s) \geq h(s) = \frac{4}{\sqrt{2 \pi}} s(1-s)
$$
and the upper bound is proven.
\end{proof}
\bigskip
\begin{proof} [\textbf{Proof of Fact \ref{secondder}}]
The upper bound follows immediately from the fact that, according to formula (\ref{derq2}) one has $q''(s) < q''(1/2) < - 2$ for all $0 < s < 1$. Let us prove the lower bound. By the symmetry of $q(s)$ around $s=1/2$, we may assume without loss of generality
that $h < 1/2$. Define
$$
f(s) = q(s) - q(h) - q'(h) (s-h)
$$
and
$$
g(s) = h^{-2} f(0) (s-h)^2.
$$
Note that by definition, the functions $f(0) = g(0)$, $f(h) = g(h) = 0$ and $f'(h) = g'(h) = 0$. Now, according to formula (\ref{derq2}), the function $q'(s)$ is convex in $[0,1/2]$ (here, we use the assumption that $h < 1/2$). Therefore, the function $w(s) = f'(s) - g'(s)$ is also convex in this interval. Now, we know that $w(h) = 0$ and that $\int_0^h w(s) ds = 0$, so from the convexity of $w(s)$ we conclude that there exists $s_0 \in (0,h)$ such that
\begin{equation} \label{whh}
w(h) = 0 \mbox{ and }  w(s) (s-s_0) \leq 0, ~~ \forall 0 < s < h
\end{equation}
and therefore
$$
\int_s^h w(x) dx \leq 0, ~~ \forall 0<s<h.
$$
It follows that $g(s) < f(s)$ for all $0 < s < h$. Moreover, since $w(s)$ is convex up to $s=1/2$, necessarily
we have $w(s) > 0$ for $h<s<1/2$ and it follows that
$$
g(s) \leq f(s), ~~ \forall 0 \leq s \leq 1/2.
$$
Next, we show that $g(s) \leq f(s)$ also for $1/2 < s < 1$, or in other words we will show that
$$
p(s) \leq q(s), ~~ \forall 0 < s < 1
$$
where
$$
p(s) = g(s) + q(h) + q'(h) (s-h).
$$
Indeed, the fact that $w(s)$ is convex up to $s=1/2$ and by (\ref{whh}), we know that $w(1/2) > 0$, which means
that $p'(1/2) < q'(1/2) = 0$, and therefore the parabola $p(s)$ attains a maximum at some point $b \leq 1/2$ which means
that $p(1-s) \leq p(s)$ for all $s < 1/2$. So by the symmetry of $q(s)$ around $s=1/2$ we get
$$
q(1-s) = q(s) \geq p(s) \geq p(1-s)  , ~~ \forall 0 < s < 1/2.
$$
We finally have $f(s) \geq g(s)$ for all $0 \leq s \leq 1$. In order to prove the lower bound, it therefore suffices
to show that
$$
- \frac{4}{h^2(1-h)^2} (s-h)^2 \leq g(s) = h^{-2} f(0) (s-h)^2 = h^{-2} (s-h)^2 (-q(h) + h q'(h))
$$
or in other words, using the assumption $h<1/2$,
$$
1 \geq q(h) - h q'(h)
$$
a combination of (\ref{derq1}) with the fact that $q(h) \leq q(1/2) < 1$ finishes the proof.
\end{proof}

\end{document}